\documentclass[12pt,a4paper,reqno]{amsart}

\usepackage[margin=2.8cm,footskip=1cm]{geometry}
\usepackage[utf8]{inputenc}
\usepackage[T1]{fontenc}
\usepackage[english]{babel}
\usepackage{amsmath}
\usepackage{amsfonts}
\usepackage{amssymb}
\usepackage{amsthm}
\usepackage{ucs}
\usepackage{graphicx}
\usepackage{xcolor}
\usepackage{dsfont}
\usepackage{tikz}
\usepackage{wasysym}
\usepackage{physics}
\usepackage{mathtools}
\usepackage{xifthen}
\usepackage{enumitem}
\usepackage{stmaryrd}
\usepackage{tikz,pgfplots}
\usepackage[hidelinks]{hyperref}
\usepackage[font={small}]{caption}

\makeatletter
\newcommand{\pushright}[1]{\ifmeasuring@#1\else\omit\hfill$\displaystyle#1$\fi\ignorespaces}
\newcommand{\pushleft}[1]{\ifmeasuring@#1\else\omit$\displaystyle#1$\hfill\fi\ignorespaces}
\makeatother

\renewcommand{\norm}[1]{\|#1\|}
\newcommand{\Bnorm}[1]{\Big\|#1 \Big\|}

\newcommand{\Z}{\mathbb{Z}}
\newcommand{\R}{\mathbb{R}}

\newcommand{\1}{\mathds{1}}

\newcommand{\cC}{\ensuremath{\mathcal C}} 
\newcommand{\cD}{\ensuremath{\mathcal D}}

\newcommand{\cL}{\ensuremath{\mathcal L}}

\newcommand{\cW}{\ensuremath{\mathcal W}}

\newcommand{\calE}{\mathcal{E}}

\newcommand{\bbN}{\mathbb{N}}

\newcommand{\bbR}{\mathbb{R}}
\newcommand{\bbS}{\mathbb{S}}
\newcommand{\bbT}{\mathbb{T}}

\newcommand{\bbZ}{\mathbb{Z}}

\def\({\left(}
\def\){\right)}
\def\gap{\mathop{\rm gap}\nolimits}

\newcommand{\entropy}{\Delta \operatorname{Ent}}

\newcommand{\free}{\mathrm{f}}
\newcommand{\periodic}{\mathrm{per}}

\theoremstyle{plain}
\newtheorem{theorem}{Theorem}[section]
\newtheorem{lemma}[theorem]{Lemma}
\newtheorem{proposition}[theorem]{Proposition}

\newtheorem{conjecture}[theorem]{Conjecture}
\newtheorem{remark}{Remark}[section]

\newtheorem{claim}{Claim}
\newtheorem{fact}{Fact}

\theoremstyle{definition}
\newtheorem{definition}{Definition}[section]
\newtheorem{obs}{Observation}
\newenvironment{Obs}{\begin{obs}}{ \end{obs}}

\author{Pietro Caputo}
\address{Universit\`a Roma Tre}
\email{pietro.caputo@uniroma3.it}

\author{S\'{e}bastien Ott}
\address{EPFL}
\email{ott.sebast@gmail.com}

\author{Assaf Shapira}
\address{MAP5, Université Paris Cité}
\email{assaf.shapira@math.cnrs.fr}

\date{\today}

\title{Relaxation time and topology in 1D $O(N)$ models}

\begin{document}

	\maketitle
	\begin{abstract}
We discuss the relaxation time (inverse spectral gap) of the one dimensional $O(N)$ model, for all $N$ and with two types of boundary conditions. We see how its low temperature asymptotic behavior is affected by the topology.
The combination of the space dimension, which here is always $1$, the boundary condition (free or periodic), and the spin state $\bbS^{N-1}$, determines the existence or absence of non-trivial homotopy classes in some discrete version. Such non-trivial topology reflects in bottlenecks of the dynamics, creating metastable states that the system exits at exponential times; while when only one homotopy class exists the relaxation time depends polynomially on the temperature. We prove in the one dimensional case that, indeed, the relaxation time is a proxy to the model's topological properties via the exponential/polynomial dependence on the temperature.
\end{abstract}	

	\section{Introduction}
	The investigation of the low-temperature behavior of classical spin systems with continuous symmetry, such as the $O(N)$ model on a lattice, is a source of many fascinating questions in equilibrium statistical mechanics \cite{friedli_velenik_2017,peled2019lectures}. For example, in the two dimensional $XY$ model, a deep understanding of the interplay between the spin wave approximation and topological aspects such as vortex formation poses significant mathematical challenges, see  \cite{dyson1956general,kosterlitz1973ordering,mermin1979topological,frohlich1981kosterlitz,frohlich1983spin} for some classical works, and see e.g.\ \cite{newman2018gaussian,garban2023statistical} for more recent studies. To delve deeper into these questions, 
	it is natural to study the Langevin dynamics associated to the $O(N)$ model, that is the reversible diffusion process with stationary distribution given by the $O(N)$ Gibbs measure. In the mean-field case,  a comprehensive analysis of the relaxation time, or inverse spectral gap, of Langevin dynamics for the $O(N)$ model,  has been achieved recently in \cite{becker2020spectral}. In particular, these results show that, when $N\geq 2$, in sharp contrast with the case of the Glauber dynamics for the Ising model ($N=1$), the relaxation time of the mean field  $O(N)$ model is at most linearly growing with the  size of the system, at any fixed temperature.   
	It is widely believed that such bounds should continue to hold for short range models as well, see e.g.\ \cite{bray1988dynamics}. In particular, it can be conjectured that for all lattice dimension $d\geq 2$, the $O(N)$ model on a lattice box with side $L$, for all $N\geq 2$, has relaxation time at most of order $L^d$ at any fixed temperature, regardless of the boundary conditions. However, even establishing that relaxation times grow at most at a polynomial rate with $L$ is a notoriously difficult open problem.

	As a much more modest objective, in this note, we explore the presence of significant topological effects in the simpler one-dimensional setting. In a one-dimensional system, it is well known that the relaxation time is of order $1$ at any fixed temperature. However, it was recently observed in \cite{cosco2021topologically} that in the one-dimensional XY model ($N=2$),  when periodic boundary conditions are imposed, as the inverse temperature $\beta$ grows logarithmically with the size $L$ of the chain, topologically induced metastable phases emerge, which correspond to distinct global winding numbers of the spin chain. When $\beta$ increases as $\log L$, the free energy barrier between these states also increases linearly with $\beta$, leading to  hitting times that are exponentially large in $\beta$. As we will see, this system exhibits relaxation times that grow exponentially with $\beta$. This phenomenon is specific of the periodic chain, and cannot occur for e.g.\ free boundary conditions. Indeed, it is a consequence of the fact that the global winding number of a periodic XY chain is a topological invariant, and two distinct phases cannot be connected by a homotopy, giving rise to a topological bottleneck. The main goal of this paper is to show that these topological effects on the dynamics are not present when $N\ge 3$. The point is that when $N\geq 3$, one can connect any two configurations of the spin chain by a continuous deformation. In particular, by using a continuous version of the canonical path method,  we will show that, as a function of $\beta$, the relaxation time of the one-dimensional  $O(N)$ model with $N\geq 3$  can grow at most polynomially in a periodic chain. We will also note that the same holds true for $O(N)$ models, this time for any $N\geq 2$, if one takes free boundary conditions instead.

	\subsection{Model and results}
	Given integers 
	$N\geq 2$, and $L\geq 2$, the one-dimensional $O(N)$ model of size $L$, with 
	free or periodic boundary conditions is defined, respectively,  by the Hamiltonians 
	\begin{gather}\label{eq:ham}
		H^{\free}_L(S) = -\sum_{i=1}^{L-1} S_i\cdot S_{i+1},\qquad 	H^{\periodic}_L(S) = -\sum_{i=1}^{L} S_i\cdot S_{i+1}.
	\end{gather}
	Here $S_i\in\bbS^{N-1}$ denotes the $i$-th spin,  $S_i\cdot S_{i+1}$ is the usual scalar product for vectors in $\R^N$, and we set $S_{L+1}\equiv S_1$ to obtain periodic boundary conditions in $H^{\periodic}_L(S)$. We write $\nu$ for 
	Lebesgue's
	measure on $\bbS^{N-1}$ and let $\nu_L $ denote the corresponding product measure on the space $\Omega_L = (\bbS^{N-1})^L$ of the spin chain configurations. Thus, the free and periodic boundary condition $O(N)$ Gibbs measure at inverse temperature $\beta>0$ is given, respectively,  by the probability measures on $\Omega_L$ defined as
	\begin{gather}\label{eq:gibbs}
		\mu_{L,\beta}^{\free}(dS) = \frac{\exp\left(-\beta H^\free_L(S)\right)}{Z^{\free}_{L,\beta}}\,\nu_L(dS),\qquad 
		\mu_{L,\beta}^{\periodic}(dS) = \frac{\exp\left(-\beta H^\periodic_L(S)\right)}{Z^{\periodic}_{L,\beta}}\,\nu_L(dS),
	\end{gather}
	where the partition function is defined, respectively, by
	\[
	Z^{\free,\periodic}_{L,\beta} = \int_{\Omega_L}\exp\left(-\beta H^{\free,\periodic}_L(S)\right)\nu_L(dS).
	\]
	The Langevin dynamics is defined as the reversible diffusion process on $\Omega_L$ with infinitesimal generator
	\begin{align}\label{eq:gen}
		\cL^{\free,\periodic}_{L,\beta}
		&= \sum_{i=1}^L \left(\frac1\beta \,D_i^2 - (D_iH_L^{\free,\periodic})\cdot D_i\right),
	\end{align}
	where $D_i$ denotes the gradient  $ \nabla_{\bbS^{N-1}}$ on the unit sphere acting on the $i$-th spin $S_i$.
	The generator \eqref{eq:gen} defines a self-adjoint operator on $L^2(\Omega_L,\mu_{L,\beta}^{\free})$
	and $L^2(\Omega_L,\mu_{L,\beta}^{\periodic})$, respectively. For any smooth function $f:\Omega_L\to \R$, the associated Dirichlet form is given by 
	\begin{gather}\label{eq:dir}
		\cD^{\free,\periodic}_{L,\beta}(f,f)  = \frac1\beta\sum_{i=1}^L\int_{\Omega_L} 
		\norm{D_i f(S)}^2
		\mu_{L,\beta}^{\free,\periodic}(dS),
	\end{gather}
	and $\|\cdot\|$ denotes the vector norm. 
	The spectral gap is defined by the variational principle
	\begin{gather}\label{eq:varprin}
		\gap^{\free,\periodic}_{L,\beta}  = 
		\inf_{f}\frac{\cD^{\free,\periodic}_{L,\beta}(f,f)  }{{\rm Var}^{\free,\periodic}_{L,\beta}(f)}\,,
	\end{gather}
	where ${\rm Var}^{\free,\periodic}_{L,\beta}(f)=\mu^{\free,\periodic}_{L,\beta}(f^2)-\mu^{\free,\periodic}_{L,\beta}(f)^2$ denotes the variance functional with respect to $\mu^{\free,\periodic}_{L,\beta}$ and $f$ ranges over all non-constant smooth functions on $\Omega_L$. The relaxation time is defined as the inverse of the spectral gap
	\begin{gather}\label{eq:trel}
		T_{\rm rel}^{\free,\periodic}(L,\beta) = \frac1{\gap^{\free,\periodic}_{L,\beta}}.
	\end{gather}
	At any fixed $\beta$ the one-dimensional nearest neighbor spin system satisfies exponential decay of covariances uniformly in the boundary conditions \cite{friedli_velenik_2017}. Then it is not difficult to prove, see e.g.\ \cite{martinelli1999lectures}, that the system has a uniformly positive spectral gap, that is  there exists a constant $C(N,\beta)$ independent of $L$, such that 
	\begin{gather}\label{eq:trel2}
		T_{\rm rel}^{\free,\periodic}(L,\beta) \leq C(N,\beta).
	\end{gather}
	Here we are interested in detecting topologically induced slowdown effects on the relaxation to equilibrium which could appear in the case where $\beta$ grows with $L$. In particular, as shown in \cite{cosco2021topologically}, these phenomena do occur in the $XY$ model ($N=2$) in the case of periodic boundary conditions, when $\beta$ is at least of order $ \log L$. 
	
	We start by showing that in the case of free boundary conditions there is no topologically induced slowdown, in the sense that the relaxation time is upper bounded as follows. 
	
	\begin{theorem}[$O(N)$ model with free boundary]\label{th:free}
		For any $N\geq 2$, $L\in\bbN$, and $\beta\geq 1$, 
		\begin{gather}\label{eq:trelfree}
			T_{\rm rel}^{\free}(L,\beta) \leq 
			C(N)\,L^2\beta^{(N+1)/2}
			,
		\end{gather}
		where $C(N)$ depends only $N$.
		In particular, for each fixed $N$, it grows at most polynomially in $\beta$.
	\end{theorem}
	We note that, in light of the bound \eqref{eq:trel2},  the above estimate becomes relevant only when $\beta$ grows with $L$. In this case, if $\beta $ is at least $C\log L$ for some large $C$, as we discuss in Section \ref{sec:discussion} below, a diffusive scaling of the relaxation time should be expected, and thus at least qualitatively,  the bound in Theorem \ref{th:free} should be tight.  
	In the case of periodic boundary conditions, in agreement with \cite{cosco2021topologically}, we show the following bounds quantifying  the  topologically induced slowdown
	for $N=2$. 
	\begin{theorem}[$XY$ model on the cycle]\label{th:periodic2}
		Let  $N=2$. For $L\in\bbN$, and $\beta\geq 1$, 
		\begin{gather}\label{eq:trelper2}
			 T_{\rm rel}^{\periodic}(L,\beta) \leq C
			 \beta^{3/2}  e^{2\beta} L^2 ,
		\end{gather}
		for some absolute constant $C$. 
		Moreover, for any $L\in\bbN$, there exists a constant $c(L)$ depending on $L$ such that for all  $\beta\geq 1$,
		\begin{gather}\label{eq:trelper3}
			 T_{\rm rel}^{\periodic}(L,\beta) \geq c(L)\beta\, e^{(2-C_0L^{-1})\beta},
		\end{gather}
		where $C_0$ is an absolute constant. 
	\end{theorem}
	The lower bound \eqref{eq:trelper3} is based on a rather crude argument and provides no meaningful $L$-dependance. 
	However, as discussed in Section \ref{sec:metastab}, the leading exponential term $e^{2\beta}$ in \eqref{eq:trelper2}-\eqref{eq:trelper3} captures the correct metastable behavior associated to the energy barrier of size $2\beta$ between states with winding number zero and states with non-zero winding number. We refer to Remark~\ref{rem:metastab} for the sketch of a finer energy-entropy argument providing quantitative $L$ dependance in the metastable regime $\beta\ge C\log L$. 

	Finally, we prove that there is no topologically induced slowdown
	for $N\geq 3$.
	\begin{theorem}[$O(N)$ model on the cycle, $N\ge 3$]\label{th:periodic3}
		For  $N\geq 3$, $L\in\bbN$, and $\beta\geq 1$, 
		\begin{gather}\label{eq:trelperN}
			T_{\rm rel}^{\periodic}(L,\beta) \leq C(L) \beta.
		\end{gather}
		Moreover, in the case of the Heisenberg model ($N=3$), one can take  $C(L)=e^{C L\log L}$ for some constant $C$ not depending on $L$ and $\beta$.
	\end{theorem}
	
	\subsection{Discussion, conjectures and open problems}\label{sec:discussion}
	We emphasize that these estimates are far from optimal and do not capture all features of relaxation to equilibrium of the spin chain. However, they are sufficient to rule out the presence of topological bottlenecks in the relaxation process for all $N\geq 2$ in the case of free boundary and for $N\geq 3$ in the case of periodic boundary. Let us give some comments on our proofs. Roughly speaking, for Theorem \ref{th:free} we use the fact that the spin chain has a product structure in the case of free boundary conditions, when one considers the ``increment'' variables $S_{i+1}-S_i$. This is achieved by a suitable change of variables that allows a convenient representation for the Hamiltonian. In  this setup, a simple tensorization argument applies and one obtains the estimate \eqref{eq:trelfree} by changing back to the original spin variables. The explicit dependance on the variable $\beta$ in \eqref{eq:trelfree} is obtained by a quantitative bound on the spectral gap for a single increment variable. 
	The upper bound in \eqref{eq:trelper2} is obtained by reducing the problem to the free boundary case via a perturbation argument. On the other hand, for the lower bound in   \eqref{eq:trelper2} we use an upper bound on the Cheeger constant. This is based on the choice of a suitable bottleneck event that was already analysed in \cite{cosco2021topologically}. 
	
	The proof of Theorem \ref{th:periodic3} requires more work. We use a continuous version of the so-called canonical path method, see e.g.\ \cite{sinclair1992improved} for a classical formulation in the discrete setting. The main idea is to construct a path consisting of a continuous, energy decreasing transformation, which allows one to move any given configuration of $L$ points on a sphere $\bbS^{N-1}$ to a configuration where all points lie in a small neighbourhood of a pole of the sphere. Once the system is confined to such a neighbourhood, convexity considerations allow us to conclude the desired statement. Checking that such a construction is possible, and controlling the entropy associated to the contracting path requires some non-trivial analysis.
	Note that this can only work in the case $N\geq 3$, since for $N=2$ it is prevented by the topological obstruction discussed above.

	Concerning dimension higher than one, in light of the above, it is natural to conjecture that if one considers the $O(N)$ model on the $d$-dimensional torus $(\bbZ/L\bbZ)^d$, then topological bottlenecks are related to the homotopy structure. More precisely, one expects the relaxation time to grow exponentially with $\beta$ when there are some non-trivial homotopy classes of maps from $\bbT^d \to \bbS^{N-1}$.
	
	Indeed, when $\beta$ is large the angle between two neighboring spins is small, and the discrete configuration of the spin system looks like a continuous field, i.e., a map from $\bbT^d$ to $\bbS^{N-1}$. With this in mind, the dynamics of the spin system corresponds to homotopy of the continuous field, and if there are several homotopy classes, then moving from one to the other requires the creation of a discontinuity. The energetic cost of this discontinuity creates the topological bottleneck.

	While this description provides a good intuition to explain the results in this paper, in a more general setting we expect such metastability to depend also on the Riemannian structure of $\bbS^{N-1}$ and not solely its topology.
	Consider for example an hourglass spin state, which has the same topology as $\bbS^2$ but non-constant curvature. As in the $O(3)$ model, the ground state is when all spins point in the same direction.
	However, this system contains a metastable state, where spins are placed on the narrow part of the hourglass, equally spaced, so that the configuration has non-zero winding.
	In this example, the energetic cost is due to the continuous transformation and not the creation of discontinuities. We therefore see that endowing the same topological sphere with a Riemannian structure other than the standard one (i.e., the homogeneous metric induced by the scalar product in $\R^3$) could change the metastability properties.

	This example suggests that homotopy classes of maps are not the correct candidate to be a marker for metastability. To remedy this, introduce the following energy functional. For $M,M'$ two compact Riemannian manifolds, and $f:M\to M'$, let $\calE(f)$ be the energy of $f$:
	\begin{equation}
		\mathcal{E}(f) = \int_M \norm{Df(x)}^2_{f(x)} \dd x,
	\end{equation}where $\norm{\ }_p$ is the norm on the tangent space of $M'$ at $p$.

	Extrema of the functional $\mathcal E$, called \emph{harmonic maps}, are well studied mathematical objects, see e.g. \cite{EellsLemaire83harmonicmaps}. Loosely stated, our general conjecture is that if one considers a model on some graph $G$ with spin taking value in the manifold $M'$, and if
	\begin{itemize}
		\item the graph $G$ ``approximates well'' the manifold $M$ (for example: the graph Laplacian on $G$ is close in some sense to the Laplacian on $M$),
		\item the Hamiltonian of the model is a ``approximate discrete version'' of the energy functional $\calE$. This is the generic behavior of nearest-neighbor attractive interaction near low energy states.
	\end{itemize}Then, the existence of metastable states is equivalent to the existence of non-trivial local minimizers of $\calE$. Specializing to the spin $O(N)$ model, one obtains the next (more precise) conjecture.
	\begin{conjecture}
		Let $M$ be a compact Riemannian manifold, and define the energy functional operating on $f:M\to \bbS^{N-1}$ as
		\begin{equation}
			\mathcal{E}(f) = \int_M \norm{Df(x)}^2 \dd x,
		\end{equation}
		where the norm is taken with respect to the standard Riemannian metric on the sphere.
		Consider the $O(N)$ model on  a graph $G$ discretizing $M$. Then the relaxation time of the corresponding Langevin dynamics grows exponentially fast with $\beta$ if and only if $\mathcal E$ has non-trivial (i.e., non-constant) local minima.
	\end{conjecture}
	In particular, for the case of $\Z^d$ with free boundary conditions, there should be no such bottlenecks, that is relaxation times growing polynomially as a function of $\beta$, for all $N\geq 2$, since here the homotopy group is always trivial and the harmonic mappings are all constant. However, we leave it as an open problem to obtain quantitative bounds in dimension $d>1$.
	
	An interesting example to study would be the case $N=4$, $M=\bbS^3$: there are several homotopy classes of continuous maps from $\bbS^3$ to $\bbS^3$ (even countably many, by the Hurewicz theorem), but the energy $\calE$ has no non-trivial local minima \cite{EellsLemaire83harmonicmaps}, so our conjecture is that there is no metastable behaviour in this case.
	
\subsection{Scaling limit in low temperature heuristics}
	We conclude this introduction with a brief informal discussion of the behavior of the system for extremely low temperature, that is when $\beta$ is large as a function of $L$, that could serve as a heuristic guide for a more precise analysis in this regime. For the sake of simplicity, we discuss the problem only for the XY model ($N=2$) and give only a brief comment on the case $N\geq 3$ afterwards. 
	When $N=2$ we may parametrize $S_i$ by a single  angle $X_i\in [-\pi,\pi]$ and if  $\beta=\beta(L)$ is very large we may assume there is a well defined  lift to $\bbR$, so that our variables are now $X_i\in\bbR$, and the center of mass $\overline{X}_L  =\frac{1}{L}\sum_{i=1}^LX_{i}$ satisfies the relation
	\begin{align}\label{eq:bari}
		\dd\overline{X}(t) & =\frac{\sqrt 2}{\beta\sqrt L}\,d B(t)\,,
	\end{align}
	where $B(t)$ denotes the standard brownian motion. To see this, observe that by definition \eqref{eq:gen}, the dynamics is given by the SDEs
	\begin{align}\label{eq:eq1}
		\dd X_i(t)  =-\partial_{X_{i}}H(X)\,\dd t+\frac{\sqrt 2}{\beta}dB_i(t)\,,\qquad i=1,\dots,L
	\end{align}
	where  the $B_i$'s are independent standard Brownian motions, and the interaction has the form
	\[
	H(X)=\sum_{i\sim j}h_{ij}\left(X_{i}-X_{j}\right)=\sum_{i\sim j}h_{ij}\left(X_{j}-X_{i}\right),
	\]
	where the sum ranges over the edges of some finite graph, and $h_{ij}=h_{ji}=\cos(\cdot)$. The graph is the segment $\{1,\dots,L\}$ in the case of free boundary conditions, and it is the $L$-cycle in case of periodic boundary.
	Therefore,  
	\[
	\sum_{i}\partial_{X_{i}}H(X)=\sum_{i}\sum_{j:\,i\sim j}h_{ij}'(X_{i}-X_{j})=-\sum_{i}\sum_{j:\,i\sim j}h_{ij}'(X_{j}-X_{i}) =0.
	\]
	In particular, \eqref{eq:bari} shows that the center of mass relaxes on a time scale proportional to $L$. Clearly, the above holds in any dimension $d\geq 1$, for the $d$-dimensional cube with side $L$ with free or periodic boundaries, provided $L$ is replaced by $L^d$.  This can be seen as the starting point to establish volume order relaxation times estimates for the low temperature XY model. However, one has to keep in mind that this center of mass motion can be interpreted as a meaningful mode of the system, namely the global phase, only when all spins
	point in approximately the same direction. As noted in \cite{cosco2021topologically}, for $d=1$, the condition $\beta \gg L$ suffices to ensure that with large probability all spins are closely aligned, that is the winding number is zero, and there is a well defined lift $X$ as above. When  $\beta \not \gg L$, 
	one needs to consider the sum of all spins, and of the corresponding Brownian motions in $\bbR^2$,  as actual vectors, which makes the analysis considerably more involved; see \cite{becker2020spectral} for a treatment of the mean field case.
	
	Beyond the center of mass discussed above, one can also consider a stochastic PDE describing the continuum limit of the full configuration of the system.
	Consider the field  $\phi:S^{1}\times\R\to S^{1}$ defined by
	\begin{align}\label{eq:defphi}
		\phi(\xi,\tau)=\frac{1}{\sqrt{L}}X_{\left\lfloor L\xi/2\pi\right\rfloor }(L^{2}\tau).
	\end{align}
	When $\beta$ is very large, we may approximate $\partial_{X_{i}}H(X)\approx X_{i}-X_{i-1}-(X_{i+1}-X_{i})\approx 4\pi^2L^{-3/2}\Delta\phi(\xi,\tau)$,
	so, with $t=L^2\tau$ one has $L^{-1/2}\partial_{X_{i}}H(X)\dd t\approx 4\pi^2\Delta\phi(\xi,\tau)\dd \tau$. Moreover, reasoning as in \cite[Section 2]{hairer2009introduction}, 
	\[
	L^{-1/2}\frac{\dd B_{\left\lfloor L\xi/2\pi\right\rfloor }(t)}{\dd t}\approx W(\xi,\tau)\,,
	\]
	where $W$ is space-time white noise on $S^1\times\bbR$. In conclusion, from \eqref{eq:eq1}, the field $\phi$ satisfies the Edwards-Wilkinson equation, or stochastic heat equation
	\begin{align}\label{eq:EW}
		\partial_\tau\phi(\xi,\tau) & \approx 4\pi^2\Delta\phi(\xi,\tau) + \sqrt 2\beta^{-1}W(\xi,\tau).
	\end{align}
	We note that the continuum approximation
	discussed above should give a valid description of the system, on suitable time scales, 
	provided $\beta$ grows at least logarithmically with $L$. 
	In particular, this suggests a diffusive time scale of order $L^2$ for relaxation, up to polylog corrections, when the system has free boundary conditions, as pointed out after Theorem \ref{th:free}. After this diffusive time scale, the system will relax by pure diffusion of the global phase. 
	If we consider periodic boundary conditions, then the situation is different, since one has to impose the 
	condition that $\phi(2\pi,\tau)-\phi(0,\tau)=\text{winding number}$, and thus the equation \eqref{eq:EW} is valid only within a given homotopy class, i.e.\ it describes the system for times much smaller than the metastable times $T_{\text{MS}}\approx e^{2\beta -\log(L)}$ detected in \cite{cosco2021topologically}, at which the XY chain changes winding number.
	On the metastable time scale $T_{\text{MS}}$ the dynamics will involve a random walk between the adjacent homotopy classes corresponding to $\pm1$ jumps of  the global winding number. 
	
	The above heuristic analysis can be in principle repeated for any $N\geq 3$, by working with the local coordinates chosen to parametrize the sphere $\bbS^{N-1}$. One gets, independently for each coordinate, a SDE for the center of mass and a stochastic heat equation as above. However, because of the dependence on the  choice of local coordinates on the manifold, the 
	interpretation of these equations is no longer obvious in this case.

	\section{Free boundary conditions}
	In this section we prove Theorem \ref{th:free}. We start by choosing appropriate coordinates.
	\subsection{Coordinates}
		 For $N=2$, we parametrize the system using angles $\theta_1,\dots,\theta_L\in [0,2\pi]$ and by setting
	\begin{equation*}
		S_i = \big(\cos(\theta_1+\dots+\theta_i), \sin(\theta_1+\dots+\theta_i)\big).
	\end{equation*}The uniform measure on $(\bbS^1)^L$ is then the image of the uniform measure on $[0,2\pi]^L$ by the above mapping.
	
	For $N\geq 3$,
	a vector $s \in  \bbS^{N-1}$ can be parametrized as
	\begin{equation*}
		s=s(\theta,v)=\cos(\theta) {\rm e}_1 + \sin(\theta) v\,,
	\end{equation*}
	$({\rm e}_1,\dots, {\rm e}_N)$ is the canonical orthonormal basis of $\bbR^N$, $\theta\in [0,\pi]$, and $v$ is a unit vector in the orthogonal complement of ${\rm e}_1$. Then, sampling $s$ uniformly on $\bbS^{N-1}$ is equivalent to sample $v$ uniformly on $\{x:\ x\cdot {\rm e}_1 = 0, \norm{x} =1 \}$, and $\theta$ proportionally to $\sin(\theta)^{N-2}$. For $\theta,v$ as before, denote $R_{v,\theta}$ the rotation matrix (in the standard basis) rotating the plane spanned by $v$ and ${\rm e}_1$ by an angle $\theta$ so that $R_{v,\theta} {\rm e}_1 = \cos(\theta) {\rm e}_1 + \sin(\theta) v$. In other words, $R_{v,\theta} $ is given by
	\begin{equation*}
		R_{v,\theta} x =
		x + \big((\cos(\theta)-1)x\cdot {\rm e}_1 - \sin(\theta)x\cdot v\big){\rm e}_1 + \big(\sin(\theta) x\cdot{\rm e}_1 + (\cos(\theta)-1)x\cdot v\big)v.
	\end{equation*}

	Let then $v_i, i=1,\dots, L$ be a sequence of uniform random variables on $\{x:\ x\cdot {\rm e}_1 = 0, \norm{x} =1 \}$, let $\theta_i, i=1,\dots, L$ be an sequence of random variables on $[0,\pi]$ with density proportional to $\sin(\theta_i)^{N-2}$. Suppose $v_i,\theta_i, i=1,\dots, L$ forms an independent family. Set
	\begin{equation*}
		R_i = R_{v_i,\theta_i},\quad S_i = R_1\dots R_i {\rm e}_1.
	\end{equation*}
	
	The first ingredient we need is the next simple Lemma.
	
	\begin{lemma}
		\label{lem:iid_uniform_SN}
		The sequence $S_i, i=1,\dots, L$ is an i.i.d. sequence of uniform random variables on $\bbS^{N-1}$.
	\end{lemma}
	\begin{proof}
		It is sufficient to check that for every realization of $R_1,\dots, R_{i-1}$, $S_i$ is uniformly distributed on $\bbS^{N-1}$. This follows directly from rotation invariance of the spherical measure and the fact that $R_i {\rm e}_1$ is uniform on $\bbS^{N-1}$.
	\end{proof}
	
	The final ingredient we will need is a control over partial derivatives of $S_i$ with respect to the angles $\theta_{j}$s and the vectors $v_j$s.
	We note here that by $\frac{\partial}{\partial v_j} S_i$ we mean the differential with respect to $v_j$ when fixing all other variables; in local coordinates it is given by the $(N-1)\times(N-1)$ Jacobian matrix, and we write $\Bnorm{\cdot}$ for the associated operator norm.
	\begin{lemma}
		\label{lem:partial_derivatives_coordinates_angle}
		Let $N\geq 3$. For every $i,j\in \{1,\dots, L\}$,
		\begin{equation}
			\Bnorm{\frac{\partial}{\partial \theta_j} S_i}\leq 1,\quad \Bnorm{\frac{\partial}{\partial v_j} S_i}\leq 4.
		\end{equation}
	\end{lemma}
	\begin{proof}
		If $j>i$, both quantities are $0$ and there is nothing to prove. Otherwise,
		\begin{equation*}
			\Bnorm{\frac{\partial}{\partial \theta_j} S_i}^2 = \Bnorm{\frac{\partial}{\partial \theta_j} R_jx}^2 = (x\cdot {\rm e}_1)^2 + (x\cdot v)^2 \leq 1.
		\end{equation*}where $x = R_{j+1}\dots R_{i}{\rm e}_1$, and we used the fact that ${\rm e}_1, v$ are orthogonal and of norm $1$. Then, for any $x\in \R^N$ and $h$ of norm $1$,
		\begin{multline*}
			\lim_{\epsilon\searrow 0} \frac{1}{\epsilon}(R_{v+\epsilon h,\theta}x - R_{v,\theta}x) =\\
			=-\sin(\theta)x\cdot h{\rm e}_1 + (\cos(\theta)-1)x\cdot h v + \big(\sin(\theta)x\cdot {\rm e}_1 + (\cos(\theta)-1)x\cdot v \big)h,
		\end{multline*} so that for $h$ of norm $1$ in the tangent space of $\{y:\ y\cdot {\rm e}_1 = 0, \norm{y} =1 \}$ at $v_j$, one has (as ${\rm e}_1, v_j, h$ are orthogonal)
		\begin{equation*}
			\Bnorm{\frac{\partial}{\partial v_j} S_i(\theta, v)}^2 =\Bnorm{\frac{\partial}{\partial v_j} R_{v_j,\theta_j}x}^2 \leq 14
		\end{equation*}where again $x = R_{j+1}\dots R_{i}{\rm e}_1$.		It follows that $\|\partial S_i/\partial v_j \|\leq \sqrt {14}\le 4$.
	\end{proof}
	
	The interest of those parametrizations lies in the following identity: the Hamiltonian in \eqref{eq:ham} becomes
	\begin{equation}
		-H^{\free}_L(S(v,\theta)) = \sum_{i=1}^{L-1} (R_1\dots R_{i+1} {\rm e}_1)\cdot (R_1\dots R_i {\rm e}_1) = \sum_{i=1}^{L-1} (R_{i+1} {\rm e}_1) \cdot {\rm e}_1 = \sum_{i=2}^{L} \cos(\theta_i),
	\end{equation}which gives a nice factorisation of the Boltzmann weight. The same identity holds for $N=2$.
	
	Therefore, we have that
	\begin{equation*}
		\int_{\Omega_{L}} f(S)e^{-\beta H^{\free}_L(S)} \nu_L(dS) 
		\propto \int d\theta dv \,f(S(\theta,v)) \sin(\theta_1)^{N-2} \prod_{i=2}^L\sin(\theta_i)^{N-2} e^{\beta\cos(\theta_i)} ,
	\end{equation*}where
	\begin{itemize}
		\item in the case $N\geq 3$, the right-hand-side integral is over $[0,\pi]^L\times (\bbS^{N-2})^L$, and we identified $\{x:\ \norm{x}=1, x\cdot {\rm e}_1 = 1 \}$ with $\bbS^{N-2}$,
		\item in the case $N=2$, the the right-hand-side integral is over $[0,2\pi]^L$ and $v$ is off the picture.
	\end{itemize}
	
	Next, for $N\geq 3$, we introduce the probability measures $\rho_1,\dots,\rho_L$ on $[0,\pi]\times \bbS^{N-2}$ given by 
	\begin{gather*}
		d\rho_1(\theta_1, v_1) \propto \nu(dv_1) \sin(\theta_1)^{N-2} d\theta_1,\\
		d\rho_i(\theta_i,v_i) \propto \nu(dv_i) \sin(\theta_i)^{N-2} e^{\beta \cos(\theta_i)} d\theta_i,\ i=2,\dots,L.
	\end{gather*}For $N=2$, we use instead measures on $[0,2\pi]$ given by
	\begin{gather*}
		d\rho_1(\theta_1) \propto d\theta_1,\\
		d\rho_i(\theta_i) \propto e^{\beta \cos(\theta_i)} d\theta_i,\ i=2,\dots,L.
	\end{gather*}
	With this notation, we rewrite the expected value of some $f:\Omega_L\mapsto\bbR$  with respect to the Gibbs measure $\mu_L^{\free}$ as
	\begin{equation}\label{eq:tensor}
		\mu_L^{\free}(f) = \otimes_{i=1}^L\rho_i
		\big(f(S(\theta,v))\big).
	\end{equation}

	\subsection{Poincar\'e inequalities for increments measures}
	
	For $f:[0,\pi]\times \{x:\ \norm{x}=1, x\cdot {\rm e}_1  = 0\}\to \R$, which maps $(\theta,v)$ to $f(\theta, v)$, denote as before $\frac{\partial}{\partial \theta}$ the partial derivative with respect to $\theta$ and $\frac{\partial}{\partial v}$ the partial derivative with respect to $v$ (so that $\frac{\partial}{\partial v} f(\theta, v)$ is a linear function from the tangent space of $\{x:\ \norm{x}=1, x\cdot {\rm e}_1  = 0\}$ at $v$ to $\R$). We also write ${\rm Var}_{P}(f)$ for the variance of $f$ with respect to the probability measure $P$.
	
	When $N\geq 3$, the measures $\rho_i$ are product measures: the product of the uniform measure on $\{x:\ \norm{x}=1, x\cdot {\rm e}_1  = 0\} \equiv \bbS^{N-2}$ and of a measure on $[0,\pi]$. We start by proving Poincar\'e inequalities for these ``elementary constituents'' which in turn imply Poincar\'e inequalities for the $\rho_i$s. We stress that we do not try to obtain the optimal constants, but we need a reasonable control over their dependency on the parameters.
	
	We first prove the bounds used for $N\geq 3$.
	\begin{lemma}
		\label{lem:Poincare_angles_Ngeq3}
		Let $a,b\geq 0$. Let $P_{a,b}$ be the probability measure on $[0,\pi]$ with density proportional to $\sin(\theta)^{a} e^{b \cos(\theta)}$. Then, for any $f:[0,\pi]\to \R$ smooth,
		\begin{equation*}
			{\rm Var}_{P_{a,b}}(f)\leq c_1(a,b) E_{a,b}\big( |f'|^2\big),
		\end{equation*}where $E_{a,b}$ is the expectation with respect to $P_{a,b}$, and
		\begin{equation*}
			c_1(a,b) = \frac{\pi^3 2^ab^{(a+1)/2}}{\int_{0}^{\sqrt{b}\pi/2}dx x^a e^{ - x^2/2}} \text{ if } b>0,\quad  c_1(a,0) = \frac{\pi^2 2(a+1)4^a}{\pi^a}.
		\end{equation*}Moreover, $c_1(a,b) \leq 8\pi^3 2^{a}b^{(a+1)/2}$ for $b\geq 1$.
	\end{lemma}
	\begin{proof}
		Start with $b>0$. Let $P=P_{a,b}$. Set $C= \int_{0}^{\pi} dx \sin^a(x) e^{b\cos(x)}$. Then,
		\begin{align*}
			{\rm Var}_{P}(f) &= \frac{1}{2C^2}\int_{0}^{\pi}dx\int_{0}^{\pi}dy (f(x)-f(y))^2 \sin^a(x)e^{b\cos(x)}\sin^a(y)e^{b\cos(y)}\\
			&\leq \frac{\pi^2}{2C^2}\int_{0}^{\pi}dx\int_{0}^{\pi}dy \int_{0}^1 dt |f'(tx+(1-t)y)|^2 \sin^a(x)e^{b\cos(x)}\sin^a(y)e^{b\cos(y)}.
		\end{align*}We can then use that on $[0,\pi]$, $g(x) = -a\ln(\sin(x))$ is convex and non-negative, therefore
		\begin{equation*}
			g(x) + g(y) \geq g(tx+(1-t) y) + g((1-t)x + t y) \geq g(tx+(1-t) y),
		\end{equation*}and so $\sin^a(x)\sin^a(y) \leq  \sin^a(tx+(1-t) y)$. Also,
		\begin{equation*}
			\cos(x) +\cos(y) -\cos(tx+(1-t)y)\leq 1,
		\end{equation*}as $\cos$ is non-increasing on $[0,\pi]$ and less or equal to $1$. Using these and changing variable to $z=tx+(1-t)y$, we obtain
		\begin{align*}
			&\int_{0}^{\pi}dx\int_{0}^{\pi}dy \int_{0}^1 dt |f'(tx+(1-t)y)|^2 \sin^a(x)e^{b\cos(x)}\sin^a(y)e^{b\cos(y)} \\
			&\leq e^{b}\int_{0}^1 dt \int_{0}^{\pi}\frac{dz}{t}\int_{0}^{\pi}dy  |f'(z)|^2 \sin^a(z) e^{b\cos(z)}\mathbf{1}_{z-(1-t)y\in [0,t\pi]} \\
			&\leq 2\pi e^{b} \int_{0}^{\pi}dz|f'(z)|^2\sin^a(z) e^{b\cos(z)}
		\end{align*}as $\int_{0}^1 dt\frac{1}{t}\int_{0}^{\pi} dy\mathbf{1}_{z-(1-t)y\in[0,t\pi]} \leq 2\pi$.
		Now,
		\begin{equation*}
			C \geq \int_{0}^{\pi}dx \sin^a(x) e^{b - bx^2/2} 
			\geq e^{b}\int_{0}^{\pi/2}dx \frac{x^a}{2^a} e^{ - bx^2/2}
			= \frac{e^{b}}{2^ab^{(a+1)/2}}\int_{0}^{\sqrt{b}\pi/2}dx x^a e^{ - x^2/2},
		\end{equation*}as $ \cos(\gamma) \geq 1-\gamma^2/2 $ for $\gamma\in [0,\pi]$, and $\sin(\gamma)\geq \gamma/2$ for $\gamma\in [0,\pi/2]$. Combining all the estimates gives the main claim. The last point follows from (for $b\geq 1$)
		\begin{equation*}
			\int_{0}^{\sqrt{b}\pi/2}dx x^a e^{ - x^2/2} \geq \int_{1}^{\pi/2}dx e^{ - x^2/2} \geq 1/8.
		\end{equation*} The case $b=0$ follows the exact same path with the lower bound $C= \int_{0}^{\pi} \sin^a(x) \geq \frac{\pi^{a+1}}{2^{2a+1}(a+1) }$.
	\end{proof}

The complete spectrum and eigenfunctions of the spherical Laplacian are known, and in particular its spectral gap is equal $N-1$ (see, e.g.,  \cite[Section 2.2.3]{bakrygentilledoux2014}).
	\begin{lemma}
		\label{lem:Poincare_unif_sphere}
		Let $N\geq 2$. Let $\nu$ be the uniform probability measure on $\bbS^{N-1}$. Then, for any $f:\bbS^{N-1}\to \R$ smooth,
		\begin{equation*}
			{\rm Var}_{\nu}(f)\leq c_{2}(N) \nu\big( \norm{Df}^2\big),
		\end{equation*}
		with  $c_{2}(N)=\frac{1}{N-1}$.
	\end{lemma}
	
	The last bound is for the XY model case ($N=2$).
	\begin{lemma}
		\label{lem:Poincare_angles_XY}
		Let $b\geq 0$. Let $P_{b}$ be the probability measure on $[0,2\pi]$ with density proportional to $e^{b \cos(\theta)}$. Then, for any smooth $2\pi$-periodic $f:\R\to \R$,
		\begin{equation*}
			{\rm Var}_{P_{b}}(f)\leq c_{3}(b) E_{b}\big( |f'|^2\big),
		\end{equation*}where $E_{b}$ is the expectation with respect to $P_{b}$, and
		\begin{equation*}
			c_{3}(b) = \frac{\pi^3 \sqrt{b}}{ \int_{-\sqrt{b}\pi}^{\sqrt{b} \pi} e^{-x^2/2}} \text{ if } b>0,\quad  c_{3}(0) = \frac{\pi^2 }{2}.
		\end{equation*}Moreover, $c_{3}(b) \leq 2\pi^3 \sqrt{b}$ for $b\geq 1$.
	\end{lemma}
	\begin{proof}
		Let $b>0$, $C= \int_{0}^{2\pi}dx e^{b\cos(x)}$. One has
		\begin{align*}
			{\rm Var}_{P_b}(f) &= \frac{1}{2C^2} \int_{0}^{2\pi}dx\int_{-\pi}^{\pi}dy(f(x)- f(x+y))^2 e^{ b(\cos(x)+\cos(x+y))}\\
			&\leq \frac{\pi^2}{2C^2} \int_{0}^1 dt \int_{0}^{2\pi}dx\int_{-\pi}^{\pi}dy|f'(x+ty)|^2 e^{ b(\cos(x)+\cos(x+y))}
		\end{align*}(we simply shifted the integration domain of $y$ to the full period $[x-\pi,x+\pi]$ which preserves the integral by periodicity). We can then use
		\begin{equation*}
			\cos(x)+\cos(x+y) -\cos(x+ty)\leq 1
		\end{equation*}as $y\in [-\pi,\pi]$, and proceed as in the proof of Lemma~\ref{lem:Poincare_angles_Ngeq3} to obtain
		\begin{align*}
			{\rm Var}_{P_b}(f) &\leq \frac{\pi^2 e^b}{2C^2} \int_{0}^1 dt \int_{0}^{2\pi}dx\int_{-\pi}^{\pi}dy|f'(x+ty)|^2 e^{ b \cos(x+ty)}\\
			&= \frac{2\pi^3 e^b}{2C^2} \int_{0}^{2\pi}dx |f'(x)|^2 e^{ b \cos(x)} = \frac{\pi^3 e^b}{C} E_b(|f'|^2)
		\end{align*}where we used periodicity in the second line. Now, as in the proof of Lemma~\ref{lem:Poincare_angles_Ngeq3}, $C \geq \int_{-\pi}^{\pi} e^{b-bx^2/2} = \frac{e^b}{\sqrt{b}} \int_{-\sqrt{b}\pi}^{\sqrt{b} \pi} e^{-x^2/2} $, which gives the claim. A simplified version of the above treats the case $b=0$ (there, $C=2\pi$).
	\end{proof}
	
	We end this section by noticing that we have everything we need to control $\mu_L^{\free} = \otimes_{i=1}^L\rho_i$, as the $\rho_i$s are themselves product of one or two of the above cases. More precisely, for $f:[0,\pi]\times \bbS^{N-2}\to \R$ smooth ($N\geq 3$), $i\geq 2$, and $(\theta,v)\sim \rho_i$, by the tensorization property of variance,
	\begin{align}
		\label{eq:tensor_var_rho}
		\nonumber{\rm Var}_{\rho_i}(f) 
		&\leq P_{N-2,\beta}\otimes \nu \big({\rm Var}_{\nu}(f(\theta,v)) + {\rm Var}_{P_{N-2,\beta}}(f(\theta,v)) \big)\\
		\nonumber&\leq P_{N-2,\beta}\otimes \nu \big(c_2(N-2)\nu(\norm{\partial_v f(\theta,v)}^2 ) + c_1(N-2,\beta)E_{N-2,\beta}(|\partial_{\theta}f(\theta,v)|^2) \big)\\
		&= c_2(N-2)\rho_i \big(\norm{\partial_v f(\theta,v)}^2 \big) +  c_1(N-2,\beta) \rho_i \big(|\partial_{\theta}f(\theta,v)|^2 \big),
	\end{align}where $P_{a,b}$ is the measure of Lemma~\ref{lem:Poincare_angles_Ngeq3}. A similar bound holds for $i=1$.
	
	\subsection{Proof of Theorem \ref{th:free} }

	Introduce
	\begin{equation}
		\label{eq:c_beta_N_def}
		c(\beta,N) = \begin{cases}
			\max(c_3(\beta),c_3(0)) & \text{ if } N=2,\\
			\max(c_1(N-2,\beta), c_1(N-2,0), c_2(N-2)) & \text{ if } N\geq 3.
		\end{cases}
	\end{equation}By Lemmas~\ref{lem:Poincare_angles_Ngeq3},~\ref{lem:Poincare_unif_sphere}, and~\ref{lem:Poincare_angles_XY}, one has that for $\beta$ larger than $1$,
	\begin{equation}
		\label{eq:asymp_c_beta_N}
		c(\beta,N)\leq \pi^32^{N+1}\beta^{(N-1)/2}.
	\end{equation}
	
	The upper bound in Theorem \ref{th:free} is a consequence of the following Lemma and of~\eqref{eq:asymp_c_beta_N}.
	\begin{lemma}
		Let $f:\Omega_L\to \R$ be smooth. Then,
		\begin{equation*}
			{\rm Var}_{\mu_L^{\free}}(f)\leq \beta \,17 c(\beta,N)
			L^2 
			\,\cD^{\free}_{L,\beta}(f,f)\,, 
		\end{equation*}
		where $c(\beta,N)$ is given by~\eqref{eq:c_beta_N_def}.
	\end{lemma}
	\begin{proof}
		Since $\mu_L^{\free}$ is the product measure \eqref{eq:tensor}, we have that by the tensorization of variance,
		\begin{equation*}
			{\rm Var}_{\mu_L^{\free}}(f) \leq \otimes_{i=1}^L\rho_i\Big(\sum_{i=1}^L {\rm Var}_{\rho_i}\big(f(\theta_1,v_1,\dots,\theta_L,v_L) \big) \Big).
		\end{equation*}Applying~\eqref{eq:tensor_var_rho}, one obtains
		\begin{equation}
			\label{eq:proof:freeBC:eq1}
			{\rm Var}_{\mu_L^{\free}}(f) \leq c(\beta,N)\otimes_{i=1}^L\rho_i\Big(\sum_{i=1}^L\big( \norm{\partial_{v_i}f}^2 + |\partial_{\theta_i} f|^2 \big)\Big).
		\end{equation}
		
		Moreover, by the chain rule, for $f:(\bbS^{N-1})^L\to \R$ smooth, one has
		\begin{equation*}
			\partial_{v_i} f\big(S(\theta,v)\big) = \sum_{j = i}^L D_jf(S(\theta)) \cdot 
			\partial_{v_i} S_j(\theta_1,v_1,\dots,\theta_j,v_j).
		\end{equation*}and
		\begin{equation*}
			\partial_{\theta_i} f\big(S(\theta,v)\big) = \sum_{j = i}^L D_jf(S(\theta)) \cdot 
			\partial_{\theta_i}S_j(\theta_1,v_1,\dots,\theta_j,v_j).
		\end{equation*}
		Observe that by our choice of coordinates (Lemma~\ref{lem:partial_derivatives_coordinates_angle}), for any $L\geq j\geq i\geq 1$,
		\begin{equation*}
				\norm{ \partial_{\theta_i}S_j(\theta_1,v_1,\dots, \theta_L,v_L) } \leq 1,\quad 
				\norm{ \partial_{v_i} S_j(\theta_1,\dots, \theta_L,v_L) }\leq 4.
		\end{equation*}
		Therefore, by Cauchy-Schwarz,
		\begin{align*}
			\sum_{i=1}^L\big( \norm{\partial_{v_i}f}^2 + |\partial_{\theta_i} f|^2 \big)&\leq \sum_{i=1}^L\Big( 16L\sum_{j=i}^L\norm{ D_jf}^2 + L\sum_{j=i}^L \norm{D_jf}^2 \Big)\\ & 
			\leq 17 L^2 \sum_{i=1}^L\norm{ D_if}^2 .
		\end{align*}
		Plugging this in~\eqref{eq:proof:freeBC:eq1}, and recalling~\eqref{eq:dir} concludes the proof.
	\end{proof}

	\section{$XY$ model on the cycle}
	In this section we prove Theorem \ref{th:periodic2}. We start with the proof of the upper bound. 
	
	\subsection{Upper bound}
	The proof is based on the upper bound for free boundary conditions in Theorem \ref{th:free} and simple comparison between the free and the periodic boundary condition system. The first observation is that from the definitions \eqref{eq:ham}-\eqref{eq:gibbs} and the fact that $|S_i\cdot S_{i+1}|\leq 1$ it follows that the relative densities $\dd\mu_L^{\free}/\dd\mu_L^{\periodic},\dd\mu_L^{\periodic}/\dd\mu_L^{\free} $ satisfy
	\begin{equation}\label{eq:relative}
		\left\|\frac{\dd\mu_L^{\free}}{\dd\mu_L^{\periodic}}\right\|_\infty\le \frac{Z^{\periodic}_{L,\beta}}{Z^{\free}_{L,\beta}}\,\exp\(\beta\),\qquad
		\left\|\frac{\dd\mu_L^{\periodic}}{\dd\mu_L^{\free}}\right\|_\infty\le \frac{Z^{\free}_{L,\beta}}{Z^{\periodic}_{L,\beta}}\,\exp\(\beta\).
	\end{equation}
	At this point, for any smooth function $f:\Omega_L\to \R$ one has 
	\begin{align*}
		{\rm Var}_{\mu_L^{\periodic}}(f)&=\inf_{c\in \bbR} 	\mu_L^{\periodic}((f - c)^2)
		\leq \mu_L^{\periodic}((f - \mu_L^{\free}(f))^2)\\&\leq  \frac{Z^{\free}_{L,\beta}}{Z^{\periodic}_{L,\beta}}\,\exp\(\beta\)\mu_L^{\free}((f - \mu_L^{\free}(f))^2)= \frac{Z^{\free}_{L,\beta}}{Z^{\periodic}_{L,\beta}}\,\exp\(\beta\){\rm Var}_{\mu_L^{\free}}(f).
	\end{align*}
	Similarly, 
	\begin{align*}
		\cD^{\free}_{L,\beta}(f,f)\leq \frac{Z^{\periodic}_{L,\beta}}{Z^{\free}_{L,\beta}}\,\exp\(\beta\)\cD^{\periodic}_{L,\beta}(f,f).
	\end{align*}
	Therefore, from Theorem \ref{th:free} it follows that
	\begin{align*}
		{\rm Var}_{\mu_L^{\periodic}}(f)\leq 
		e^{2\beta}C(2)L^2\beta^{3/2} 
		\cD^{\periodic}_{L,\beta}(f,f).
	\end{align*}
	This implies the upper bound \eqref{eq:trelper2}.
	
	\subsection{Lower  bound}\label{sec:metastab}
	We use the variational principle \eqref{eq:varprin}. In order to construct an appropriate test function, we follow \cite{cosco2021topologically}, where a proxy for the winding number of the spin chain was defined as follows. 
	Let $S_i\in\bbS^1$ denote the $i$-th spin, with $S_{L+1}=S_1$, and write  $[\theta]$ for the representative in the interval $(-\pi,\pi]$ of any $\theta\in \bbS^1$. 
	If $S\in(\bbS^1)^L$ is such that $S_{i+1} - S_{i}\in\bbS^1\setminus\{\pi\}$ for all $i=1,\dots,L$, define the function 
	\begin{align}\label{eq:wind}
		\cW(S) = \frac1{2\pi}\sum_{i=1}^L [S_{i+1} - S_{i}].
	\end{align}
	Because of the periodic boundary condition, $\cW$ is an integer, which can be interpreted as the winding number of the spin configuration. 
	Moreover, the function $\cW$ is continuous in its domain of definition $D$ given by 
	\begin{gather}\label{eq:adb}
		D = \left\{	S\in(\bbS^1)^L: \; S_{i+1} - S_{i}\in\bbS^1\setminus\{\pi\}, \;	i=1,\dots,L\right\}\,.
	\end{gather}
	Next define the events
	\begin{gather}\label{eq:adb2}
		B = \{S\in D:\;\cW(S)=0\}\,,\\
		A_\delta = \{	S\in D: \; \cW(S)=0\; \,\text{and } S_{i}-S_{i+1}\in[\pi-\delta,\pi+\delta]\text{ for some }i \}\,.
	\end{gather}
	We let $f: (\bbS^1)^L\mapsto\bbR$ denote a $\cC^\infty$ function such that $f=0$ on $B^c$, $f=1$ on $B\setminus A_\delta$ and such that $\|D_i f\|_\infty \leq C_\delta$, where $C_\delta=O(1/\delta)$ is a constant independent of $\beta, L$.  Since $ D_i f(S)\neq 0$ implies that $S\in A_\delta$, using this function $f$ in the variational principle \eqref{eq:varprin} one obtains
	\begin{gather}\label{eq:varprin21}
		\gap^{\periodic}_{L,\beta}  \leq \frac1\beta\,C_\delta^2 L\,
		\frac{\mu^{\periodic}_{L,\beta}(A_\delta)}{{\rm Var}^{\periodic}_{L,\beta} f}.
	\end{gather}
		Moreover, by definition of $f$ one has 
	\begin{gather}\label{eq:varprin22}
		{\rm Var}^{\periodic}_{L,\beta} f
		\geq \mu^{\periodic}_{L,\beta}(B\setminus A_\delta)\mu^{\periodic}_{L,\beta}(B^c).
	\end{gather}
	It remains to estimate $\mu^{\periodic}_{L,\beta}(A_\delta)$ and the probabilities in \eqref{eq:varprin22}.
	
	For $\delta>0$, define the events 
	\begin{align*}
		B_{\delta}^{0} & =\left\{ S_{i}-S_{i}^{0}\in[-\delta,\delta]\text{ for all }i=1,\dots,L\right\} ,\\
		B_{\delta}^{1} & =\left\{ S_{i}-S_{i}^{1}\in[-\delta,\delta]\text{ for all }i=1,\dots,L\right\} ,
	\end{align*}
	where the configurations $S^{0}, S^{1}$, seen as variables in the complex plain, are defined by
	\begin{align*}
		S_{j}^{0}  =1\,, \qquad S_{j}^{1}  =e^{2\pi i\frac{j}{L}}\,, \qquad j\in\{1,\dots L\}.
	\end{align*}
	Observe that, for small enough $\delta=\delta(L)$ depending on $L$, 
	\[
	B_{\delta}^{0}\in\{\mathcal{W}=0\}\text{ and }B_{\delta}^{1}\in\{\mathcal{W}=1\},
	\]
	and, writing $H(S):=
	H^{\periodic}_L(S)$,  
	\begin{align*}
		H(S)
		 & \ge2-\delta^{2}-L,\qquad S\in A_{\delta},\\
		H(S) & \le-(1-\delta^{2})L,\qquad S\in B_{\delta}^{0},\\
		H(S) & \le-L+C(\delta+L^{-1}) + C\delta^2 L\,,
		\qquad S\in B_{\delta}^{1},
	\end{align*}
	for some absolute constant $C>0$.
	With the notation $\text{vol}(A)=\int_A \dd S$ for any $A\subset \Omega$, $\Omega=\Omega_L$, we obtain 
	\[
	\text{vol}(B_{\delta}^{0})e^{\beta(1-\delta^{2})L}\le\int_\Omega e^{-\beta H(S)}\dd S\le\text{vol}(\Omega)e^{\beta L}.
	\]
	Therefore, writing $\mu=\mu^{\periodic}_{L,\beta}$, 
	\begin{align*}
		\mu(A_{\delta})&=\frac{\int_{A_{\delta}}e^{-\beta H(S)}\dd S}{\int_{\Omega}e^{-\beta H(S)}\dd S}\le\frac{\text{vol}(A_{\delta})e^{-\beta\left(2-\delta^{2}-L\right)}}{\text{vol}(B_{\delta}^{0})e^{\beta(1-\delta^{2})L}}\\
		&=\frac{\text{vol}(A_{\delta})}{\text{vol}(B_{\delta}^{0})}\,e^{-\beta\left(2-\delta^{2}(L+1)\right)}=C_{\delta,L}\,e^{-\beta\left(2-\delta^{2}(L+1)\right)},
	\end{align*}
	where $C_{\delta,L}$ is a constant depending on $\delta,L$.
	We can also estimate 
	\[
	\mu(B_{\delta}^{0})=\frac{\int_{B_{\delta}^{0}}e^{-\beta H(S)}\dd S}{\int_{\Omega}e^{-\beta H(S)}\dd S}\ge\frac{\text{vol}(B_{\delta}^{0})e^{\beta(1-\delta^{2})L}}{\text{vol}(\Omega)e^{\beta L}}=c_{\delta,L}e^{-\delta^{2}L\beta},
	\]
	where $c_{\delta,L}$ is another constant depending on $\delta,L$.
	Finally, 
	\[
	\mu(B_{\delta}^{1})=\frac{\int_{B_{\delta}^{1}}e^{-\beta H(S)}\dd S}{\int_{\Omega}e^{-\beta H(S)}\dd S}\ge\frac{\text{vol}(B_{\delta}^{1})e^{\beta\left(L-C(\delta+L^{-1}) - C\delta^2 L
	\right)}}{\text{vol}(\Omega)e^{\beta L}},
	\]
	which implies
	\[
	\mu(B_{\delta}^{1})\ge c_{\delta,L}e^{-\left(C(\delta+L^{-1}) + C\delta^2 L
	\right)\beta},
	\]
	for some other constant $c_{\delta,L}$ depending on $\delta,L$. Summarizing, we obtain the following estimate.  

	\begin{proposition}\label{prop:crude}
		For all $L\in\bbN$, $\delta=\delta(L)\le L^{-1}$, for all $\beta\ge1$,
		\begin{equation}\label{eq:bottleneck}
			\frac{\mu(A_{\delta})}{\mu^{\periodic}_{L,\beta}(B\setminus A_\delta)\mu^{\periodic}_{L,\beta}(B^c)}
			\le\frac{\mu(A_{\delta})}{\mu(B_{\delta}^{0})\mu(B_{\delta}^{1})}\le C_{\delta,L}e^{-(2-C_0L^{-1}
			)\beta},
		\end{equation}
		where $C_{\delta,L}$ is a constant depending on $\delta,L$, and $C_0$ is an absolute constant.
	\end{proposition}
	The lower bound in \eqref{eq:trelper3} now follows from \eqref{eq:varprin21}-\eqref{eq:varprin22} and Proposition \ref{prop:crude}.

	\begin{remark}[A finer lower bound for $\beta\ge C\log L$]\label{rem:metastab}
		The estimate in Proposition \ref{prop:crude} captures the correct energy barrier of size $2\beta$ up to $O(L^{-1})$ corrections, see the matching upper bound \eqref{eq:trelper2}. However, it provides no quantitative estimate in terms of the system size $L$. 
		To resolve this, we observe that Proposition \ref{prop:crude} can be considerably refined by adapting the analysis from \cite{cosco2021topologically}. 
			Rather than giving an explicit derivation, we content ourselves with the following observations. One can use the arguments in \cite[Section 4]{cosco2021topologically} to prove that for all $L$ and $\beta\geq 1$, 
			\[
			\mu^{\periodic}_{L,\beta}(B) \geq \frac1{L}\,,\qquad \mu^{\periodic}_{L,\beta}(B^c)\geq \frac{1}{C_0\sqrt{\beta}L}e^{-C_0\beta/L}\,,
			\]
			where $C_0$ is an absolute constant, and that for any $\delta\in(0,1)$ one has 
			 \[
		\mu^{\periodic}_{L,\beta}(A_{\delta})\le C_0 
		\beta L\delta e^{-\beta(2-\delta^{2}/2)}.
		\]
		From these bounds it is not difficult to check that for an appropriate absolute constant $C_0$, if $\beta\geq C_0\log L$, $L\geq C_0$, then for all fixed $\delta\in(0,1)$, the left hand side of \eqref{eq:bottleneck} is bounded from above by 
		\[
		C_0\beta^{3/2}L^3\delta\,e^{-\beta(2-\delta^{2}/2-C_0L^{-1})}.
		\]
		Thus, from \eqref{eq:varprin21} and \eqref{eq:varprin22} one arrives at the following relaxation time lower bound: for all $\delta\in(0,1)$, there exists a constant $c_\delta>0$ such that for $L\geq 1/c_\delta$, and $\beta\geq C_0\log L$,  for an appropriate absolute constant $C_0$, one has    
		 \begin{equation}\label{eq:trelbottle}	 
		 T_{\rm rel}^{\periodic}(L,\beta) \geq c_\delta\beta^{-1/2}L^{-4}\, 
		 e^{(2-\delta^2)\beta}\,.
		\end{equation}
	\end{remark}

	\section{$O(N)$ on the cycle}
	Here we prove Theorem \ref{th:periodic3}.
	In order to simplify notation we omit some subscripts and superscripts; we only discuss in this section the relaxation time for fixed $N\ge 3$, and fixed $L$ and $\beta$, with periodic boundary conditions.

	\subsection{General strategy}
	In order to bound the relaxation time from above, we will show how to bring a configuration, using a canonical path method, to a small neighborhood where the hamiltonian in convex.
	
	\begin{definition}
		The \emph{arctic} is defined as the set of configurations where all spins are in the ball of radius $\arccos(0.99) \approx 0.02 \times 2\pi$ around ${\rm e}_1$:
		\begin{equation}
			A = \{S \in \bbS^{N-1} : S\cdot {\rm e}_1 > 0.99\}^L.
		\end{equation}
	\end{definition}
	The following is a simple consequence of our definitions. 
	\begin{fact}
		$H^{\periodic}_L$ is convex on $A$.
	\end{fact}
	
	\begin{definition}
		Fix an open set $U\subseteq \Omega$. A path from $U$ is a continuous function $\Phi:U \times [0,1] \to \Omega$.
		We define the path's {\em energy} and {\em entropy} respectively as
		\begin{align}
			\Delta H_\Phi (S,t) &= H^{\periodic}_L(\Phi(S,t))-H^{\periodic}_L (S), \\
			\entropy_\Phi (S,t) &= \log \left | \frac{d(\Phi(\cdot,t))_{\#} \nu_L}{d \nu_L} \right |,
		\end{align}
		where $\left | \frac{d(\Phi(\cdot,t))_{\#} \nu_L}{d \nu_L} \right |$ is the density of the push-forward measure of $\nu_L$ with respect to $\nu_L$; in local coordinates it is given by the determinant of the Jacobian matrix.
		The \emph{free energy barrier} of the path is defined by
		\begin{equation}
			\Delta F [\Phi] = \sup_{t\in [0,1]} \sup_{S \in U} \, \Delta H_{\Phi} (S,t) - \beta^{-1} \entropy_{\Phi} (S,t).
		\end{equation}
		The {\em speed} of the path is defined by
		\begin{equation}
			v[\Phi] = \sup_{t\in [0,1]} \sup_{S \in U} \norm{ \partial_t{\Phi}(S,t) },
		\end{equation}
		where $\norm{\cdot}$ denotes the euclidean norm
		\[
		\norm{\xi}^2 = \sum_{j=1}^L\norm{\xi_j}^2\,,\qquad \xi=(\xi_1,\dots,\xi_L)\in\bbR^L.
		\]
		 The paths to be considered below are almost everywhere differentiable, so the above is well defined as an essential supremum. 
	\end{definition}
	
	\begin{lemma}\label{lem:pathargument}
		Let $U_1,\dots,U_K$ be some finite collection of open subsets of $\Omega$, and $\Phi_1,\dots,\Phi_K$ a collection of paths, $\Phi_i : U_i \times [0,1]\to \Omega$.
		Assume:
		\begin{enumerate}
			\item $\bigcup_{i=1}^K U_i = \Omega$.
			\item $\Phi_i (S,0) = S$ and $\Phi_i (S,1) \in A$ for all $S\in U_i$.
			\item $\Delta F := \sup_i \Delta F[\Phi_i]$ and $v := \sup_i v[\Phi_i]$ are both finite.
		\end{enumerate}
		Then 
		\begin{equation}
			T_{\rm rel} \le 3 K^2 \beta (v^2e^{\beta \Delta F}+e^{2\beta \Delta F}\mu(A))
		\end{equation}
		In particular, if the energy barrier $\sup_i \sup_{t\in [0,1]} \sup_{S \in U_i} \, \Delta H_{\Phi_i} (S,t)$ is non-positive then the $\beta$-dependence of the relaxation time grows at most linearly.
	\end{lemma}
	
	\begin{proof}
		Fix some test function $f : \Omega \to \R$, and use the notation $\Phi_i(X)=\Phi_i(X,1)$. The variance of $f$ with respect to $\mu:=\mu_{L,\beta}^{\periodic}$ satisfies 
		\begin{multline}
			2 {\rm Var}_{L,\beta}^{\periodic}
			 (f) = \int d\mu(X)d\mu(Y) (f(X)-f(Y))^2\\
			\le \sum_{i=1}^K \sum_{j=1}^K  \int d\mu(X)d\mu(Y) \1_{U_i}(X) \1_{U_j}(Y) (f(X)-f(Y))^2 \\
			= \sum_{i=1}^K \sum_{j=1}^K  \int d\mu(X)d\mu(Y) \1_{U_i}(X) \1_{U_j}(Y) \\ \times (f(X)-f(\Phi_i(X)) + f(\Phi_i(X)) - f(\Phi_j(Y)) +f(\Phi_j(Y))-f(Y))^2\\
			\le 3 \sum_{i=1}^K \sum_{j=1}^K  \int d\mu(X)d\mu(Y) \1_{U_i}(X) \1_{U_j}(Y) (f(X)-f(\Phi_i(X)))^2 \\
			+3 \sum_{i=1}^K \sum_{j=1}^K  \int d\mu(X)d\mu(Y) \1_{U_i}(X) \1_{U_j}(Y) (f(\Phi_i(X))-f(\Phi_j(Y)))^2 \\
			+3 \sum_{i=1}^K \sum_{j=1}^K  \int d\mu(X)d\mu(Y) \1_{U_i}(X) \1_{U_j}(Y) (f(\Phi_j(Y))-f(Y))^2 \\
			\le 6K \sum_{i=1}^K \int d\mu(X)\1_{U_i}(X) (f(X)-f(\Phi_i(X)))^2 \\
			+ 3 \sum_{i=1}^K \sum_{j=1}^K  \int d\mu(X)d\mu(Y) \1_{U_i}(X) \1_{U_j}(Y) (f(\Phi_i(X))-f(\Phi_j(Y)))^2
			\\ = 6K\sum_{i=1}^K{\rm (I)}_i + 3\sum_{i=1}^K\sum_{j=1}^K{\rm (II)}_{i,j}
		\end{multline}
		
		We start with the first term.
		For any $i$ and $X\in \Omega$,
		\begin{align}
			(f(X)-f(\Phi_i(X)))^2 &= \left(\int_0^1 dt \sum_{j=1}^LD_jf(\Phi_i(X,t))\cdot (\partial_t \Phi_i(X,t))_j) \right)^2 \\
			&\le v^2 \int_0^1 dt \sum_{j=1}^L\norm{D_jf(\Phi_i(X,t))}^2,
		\end{align}
		where $D_j$ denotes the gradient on the sphere $\bbS^{N-1}$ acting on the $j$-th spin. 
		Plugging this into (I), recalling the definition of $\mu$, and then changing variable to $X' = \Phi_I(X,t)$, yield
		\begin{align}
			{\rm (I)}_i &\le v^2 \int_0^1 dt  \int d\mu(X)\1_{U_i}(X) \sum_{j=1}^L\norm{D_jf(\Phi_i(X,t))}^2 \\
			&= \frac{v^2}{Z} \int_0^1 dt  \int d\nu(X)
			e^{-\beta H(\Phi_i(X,t)) + \beta \Delta H_{\Phi_i}(X,t)}
			\1_{U_i}(X)\sum_{j=1}^L \norm{D_jf(\Phi_i(X,t))}^2 \\
			&= \frac{v^2}{Z} \int_0^1 dt  \int d\nu(X') e^{-\entropy_{\Phi_i}(X,t)}
			e^{-\beta H(X')} e^{ \beta \Delta H_{\Phi_i}(X,t)}
			\1_{U_i}(X)\sum_{j=1}^L \norm{D_jf(X')}^2 \\
			&\le v^2 e^{\beta \Delta F} \int_0^1 dt  \int d\mu(X')
			\sum_{j=1}^L\norm{D_jf(X')}^2 = v^2 e^{\beta \Delta F} \sum_{j=1}^L \mu (\norm {D_jf}^2) \\
			& = \beta v^2 e^{\beta \Delta F} \cD_{L,\beta}^{\periodic}(f,f). \label{eq:46}
		\end{align}
		
		For the second term, the change of variables $X' = \Phi_i(X,t) , Y'=\Phi_j(Y,t)$ leads using the same calculation to
		\begin{equation}
			{\rm (II)}_{i,j}
			\le e^{2\beta \Delta F} 
			\int d\mu(X')d\mu(Y') \1_{A}(X') \1_{A}(Y') (f(X')-f(Y'))^2.
		\end{equation}
		This last integral equals $\mu(A)^2$ times the variance of $f$ under the conditional measure $\mu(\cdot | A)$. Since on $A$ the Hamiltonian is convex, and since $\Omega$ has positive curvature $1$, the Brascamp-Lieb inequality for $\mu(\cdot | A)$, see e.g.\ \cite[Theorem 1.2]{kolesnikov2017brascamp},  tells us that:
		\begin{equation}
			\frac{1}{2} \int d\mu(X|A)d\mu(Y|A)(f(X)-f(Y))^2 \le  \sum_{j=1}^L\mu(\norm{D_jf}^2 | A).
		\end{equation}
		We conclude that
		\begin{equation}
			{\rm (II)}_{i,j} \le 2 e^{2\beta \Delta F} \mu(A) \sum_{j=1}^L\mu(\norm{D_jf}^2 ) =2 \beta e^{2\beta \Delta F} \mu(A)  \cD_{L,\beta}^{\periodic}(f,f).
		\end{equation}
	Together with the bound \eqref{eq:46}, this ends the proof of the lemma.
	\end{proof}

	\subsection{Constructing the path: a soft argument for any $N\geq 3$}
	We will construct paths to be used in Lemma \ref{lem:pathargument} in three parts: first, we use the deterministic flow in order to bring the configuration near a critical point, where all spins are on a single {\em great circle}, defined as the intersection of the sphere with a $2$ dimensional plane. Then, we pull the spins to a point perpendicular to that great circle. Finally, we rotate the configuration to the arctic.
	
	\subsubsection*{Step 1: Approaching a great circle}
	
	Define the deterministic flow $\varphi : \Omega \times \R \to \Omega$ as the solution, for each $S\in\Omega$, of the differential equation
	\begin{align}
	\frac{\partial\varphi(S,t)}{\partial t} & =-DH(\varphi(S,t)),\label{eq:deterministic_flow}\\
	\varphi(S,0) & =S,
	\end{align}	
	where we use $D=(D_j)_{j=1,\dots,L}$ for the $L$-vector of sphere gradients $D_j$.  
	\begin{claim}
	Let $S$ be a stationary point of the flow. Then the spins $S_{1},\dots,S_{L}$
	all belong to the same great circle.
	\end{claim}
	
	\begin{proof}
	By explicit calculation in $\R^{N}$,
	\begin{align*}
	-D_{i}H(S) & =D_{i}\left[S_{i-1}\cdot S_{i}+S_{i}\cdot S_{i+1}\right]=S_{i-1}-(S_{i-1}\cdot S_{i})S_{i}+S_{i+1}-(S_{i+1}\cdot S_{i})S_{i}\\
	 & =S_{i-1}+S_{i+1}-(S_{i-1}\cdot S_{i}+S_{i+1}\cdot S_{i})S_{i}.
	 \end{align*}
	Therefore, $D_{i}H(S)=0$ implies that $S_{i+1}$ is in the linear span of  $S_{i-1}$ and $S_{i}$. It follows that if $DH(S)=0$, then this holds for all $i$, and therefore all the spins belong to the same 2D plane. 
	\end{proof}
	
	\subsubsection*{Step 2: Pulling towards a single point}
	
	Once all spins are on a great circle, we move them to a point $s$ perpendicular to that circle. First, we define the flow $\phi$ that does that. We then need to show that the energy is decreasing along the flow.
	This is done in two parts---initially, spins that were on one side of the great circle move closer to it, and the others move away. After a short period, the spins on the "wrong" side cross the great circle, and in the second part of the motion all spins are on a single hemisphere. 
	Claim \ref{claim:moving_to_hemisphere} shows that if we start $\varepsilon$-close to a great circle, after time $2\varepsilon$ all points are going to be at the same side of the circle.\\
	We would then like to show that the energy is decreasing up to time $2\varepsilon$. This unfortunately is not always true---if all spins are in the hemisphere opposite to $s$, then initially distances grow and the energy increases. This problem, however, could be easily solved by replacing $s$ with $-s$, which is now in the right hemisphere. In Claim \ref{claim:up_or_down} we show that whenever we start $\varepsilon$-close to the great circle, pulling the spins to $s$ or $-s$ results in energy decrease.\\
	Finally, Claim \ref{claim:pulling_in_hemisphere} shows that once all spins are in the same hemisphere as $s$, the energy is decreasing.

	\begin{definition}
		Fix two spins $s,s'\in\bbS^{N-1}$. We define the function $\phi_{s}:\bbS^{N-1} \times\R\to\bbS^{N-1}$
		as the flow on the sphere given by the following equation in $\R^{N}$:
		\begin{align*}
			\frac{\partial}{\partial t}\phi_{s}(s',t) & =s-(\phi(s',t)\cdot s)\phi(s',t),\\
			\phi_{s}(s',0) & =s'.
		\end{align*}
		For a configuration $S=(S_{1},\dots,S_{L})$ we write $\phi_{s}(S,t)=\left(\phi_{s}(S_{1},t),\dots,\phi_{s}(S_{L},t)\right)$.
	\end{definition}

	\begin{Obs}
		If $s'$ is perpendicular to $s$, we can solve explicitly 
		\[
		\phi_{s}(s',t)=\tanh(t)\,s+\frac{1}{\cosh(t)}\,s',
		\]
		which shows that $\lim_{t\to\infty}\phi_{s}(s',t)=s$. This last fact
		is true in general---we can always choose $\tilde{s}$ and $t_{0}$
		such that
		\[
		\phi_{s}(s',t)=\tanh(t-t_{0})\,s+\frac{1}{\cosh(t-t_{0})}\,\tilde{s},
		\]
		by taking $\tilde{s}$ perpendicular to $s$ in the plane spanned
		by $s$ and $s'$, and determine $t_{0}$ using $\phi_{s}(s',0)=s'$.
	\end{Obs}

	\begin{claim}\label{claim:moving_to_hemisphere}
		For any $\varepsilon<\frac{1}{2}$, if $\left|s\cdot s'\right|<\varepsilon$
		then $\phi_{s}(s',2\varepsilon)\cdot s>0.$
	\end{claim}

	\begin{proof}
	Note that 
	\[
	\frac{\dd}{\dd t}\left(\phi_{s}(s',t)\cdot s\right)=1-(\phi(s',t)\cdot s)^{2}\ge0,
	\]
	so it is enough to prove that $\phi_{s}(s',t)\cdot s$ cannot remain
	in the interval $[-\varepsilon,0]$ for all $t\in[0,2\varepsilon]$.
	But assuming it does, 
	\[
	\frac{\dd}{\dd t}\left(\phi_{s}(s',t)\cdot s\right)=1-(\phi(s',t)\cdot s)^{2}>\frac{3}{4},
	\]
	and therefore $\phi_{s}(s',2\varepsilon)\cdot s>-\varepsilon+\frac{3}{4}\cdot2\varepsilon=\frac{\varepsilon}{2}>0$.
	\end{proof}

	\begin{claim}
		Fix $s\in\mathbb{S}^{N-1}$, and a configuration $S$ where all spins
		are on a great circle perpendicular to $s$. If not all spins are aligned,
		then
		
		\[
		\left.\frac{\dd^{2}}{\dd t^{2}}H(\phi_{s}(S,t))\right|_{t=0}<0.
		\]
	\end{claim}

	\begin{proof}
	Since all spins are perpendicular to $s$, we can write explicitly,
	setting $\underline{s}=(s,\dots,s)\in \Omega$,
	\begin{align*}
	\phi_{s}(S,t) & =\tanh(t)\,s+\frac{1}{\cosh(t)}\,S,\\
	H(\phi_{s}(S,t)) & =\frac{1}{2}\sum_{i}\norm{\phi_{s}(S,t)_{i}-\phi_{s}(S,t)_{i-1}}^{2} + {\rm const.}\\
	 & =\frac{1}{2}\frac{1}{\cosh(t)^{2}}H(S) + {\rm const.}\\
	\frac{\dd}{\dd t}H(\phi_{s}(S,t)) & =-\frac{\sinh(t)}{\cosh(t)^{3}}H(S)\\
	\frac{\dd^{2}}{\dd t^{2}}H(\phi_{s}(S,t)) & =\frac{2\sinh(t)^{2}-1}{\cosh(t)^{4}}H(S).
	\end{align*}
	\end{proof}
	
	In the following we use $d(\cdot,\cdot)$ to denote the geodesic distance on the sphere.
	\begin{claim}\label{claim:up_or_down}
	Fix $s\in\mathbb{S}^{N-1}$, and a non-constant configuration $S^{0}$
	where all spins are on a great circle perpendicular to $s$. There
	exists $\varepsilon$ small enough (depending on $S^{0}$), such that
	for all $S$ satisfying $d(S_{i},S_{i}^{0})<\varepsilon$ for all
	$i$, either $H(\phi_{s}(S,t))$ or $H(\phi_{-s}(S,t))$ (or both)
	is decreasing for $t\in[0,2\varepsilon]$.
	\end{claim}
	
	\begin{proof}
	We consider two cases. The first, is when $\frac{\dd}{\dd t}H(\phi_{s}(S,0))\le0$.
	Since $\norm{\frac{\partial}{\partial t}\phi_{s}(s',t)}\le1$, during
	the time interval $[0,2\varepsilon]$ all spins of $S$ remain at
	a distance at most $3\varepsilon$ form $S^{0}$. By the last claim
	$\frac{\dd^{2}}{\dd t^{2}}H(\phi_{s}(S^{0},0))<0$, hence we may choose
	$\varepsilon$ small enough, such that $\frac{\dd^{2}}{\dd t^{2}}H(\phi_{s}(S',t))<0$
	for all $S'$ in a $3\varepsilon$-neighborhood of $S^{0}$. Then
	$\frac{\dd}{\dd t}H(\phi_{s}(S,t))\le0$ for all $t\in[0,2\varepsilon]$.
	
	In the second case $\frac{\dd}{\dd t}H(\phi_{s}(S,0))\ge0$, but then by
	symmetry of the flow $\frac{\dd}{\dd t}H(\phi_{-s}(S,0))\le0$ and
	we are back to the first case.
	\end{proof}
	
	\begin{claim}\label{claim:pulling_in_hemisphere}
	Fix $s\in\mathbb{S}^{N-1}$ and a configuration $S$, such that $s\cdot S_{i}>0$
	for all $i$. Then $H(\phi_{s}(S,t))$ is a decreasing function for
	$t\in[0,\infty]$.
	\end{claim}
	
	\begin{proof}
	For all $i$, denoting $\underline{s}=(s,\dots,s)\in(\bbS^{N-1})^{L}$,
	\begin{align*}
	&-D_{i}H(S)\cdot(\underline{s}-(S\cdot\underline{s})S)_{i}  =  S_{i-1}\cdot(s-(S_{i}\cdot s)S_{i})+S_{i+1}\cdot(s-(S_{i}\cdot s)S_{i})\\
	 & \qquad  -(S_{i-1}\cdot S_{i}+S_{i+1}\cdot S_{i})\,S_{i}\cdot(s-(S_{i}\cdot s)S_{i})\\
	 & \qquad =  S_{i-1}\cdot s-(S_{i}\cdot s)(S_{i}\cdot S_{i-1})+S_{i+1}\cdot s-(S_{i}\cdot s)(S_{i}\cdot S_{i+1}).
	\end{align*}
	Therefore,  
	\begin{align*}&\sum_{i}-D_{i}H(S)\cdot(\underline{s}-(S\cdot\underline{s})S)_{i}  
	\\ & \qquad =2\Big(\sum_{i}S_{i}\Big)\cdot s-\Big(\sum_{i}(S_{i}\cdot S_{i-1})S_{i}\Big)\cdot s-\Big(\sum_{i}(S_{i}\cdot S_{i+1})S_{i}\Big)\cdot s\\
	 & \qquad =  \sum\left((1-(S_{i}\cdot S_{i-1}))S_{i}\right)\cdot s+\sum\left((1-(S_{i}\cdot S_{i+1}))S_{i}\right)\cdot s>  0.
	\end{align*}
	Note that if $s\cdot S_{i}>0$ then at all positive times $\phi_{s}(S,t)>0$
	(since spins always get closer to $s$). Therefore we can plug $\phi_{s}(S,t)$
	for $S$ in the above inequality, and obtain:
	\begin{align*}
	\frac{\dd}{\dd t}H(\phi_{s}(S,t)) & =DH\cdot\frac{\partial\phi_{s}(S,t)}{\partial t}\\
	 & =DH\cdot\left(\underline{s}-(\phi(S,t)\cdot\underline{s})\phi(S,t)\right) <0.
	\end{align*}
	\end{proof}

	\subsubsection*{Step 3: Combining all parts of the path}
	
	\begin{definition}
	Fix $T>0$, $s\in\bbS^{N-1}$, and a configuration $S^{0}$ where all
	spins are on a plane perpendicular to $s$. Let $U_{T,s,S^{0}}$ be
	the set of configurations $S$, such that:
	\begin{enumerate}
	\item If $S^0$ is constant, then under the deterministic flow $\varphi$, all spins $\varphi(S,T)_i$ are at distance at most $\arccos(0.99)$ from $S^0$.
	\item If $S^0$ is not constant
	\begin{enumerate}
	\item All spins $\varphi(S,T)_{i}$
	are at distance at most $\varepsilon$ from $S_{i}^{0}$; and
	\item $H(\phi_{s}(S,t))$ is decreasing for $t\in[0,2\varepsilon]$;
	\end{enumerate}
	for $\varepsilon$ given in Claim \ref{claim:up_or_down}.
	\end{enumerate}
	\end{definition}
	
	\begin{lemma}
		Fix $T,s,S^{0}$ as in definition above. Then there exists a map $\Psi_{T,s,S^{0}}:U_{T,s,S^{0}}\times[0,1]\to \Omega$
		such that:
		\begin{enumerate}
		\item $\Psi_{T,s,S^0}(S,0)=S$ and $\Psi_{T,s,S^0}(S,1)$ is in in the arctic $A$.
		\item $H(\Psi(S,t))$ is non-increasing for all $t$.
		\end{enumerate}
		Moreover, $v[\Psi_{T,s,S^0}]$ and $\sup_{S,t} |\entropy_{\Psi}(S,t)|$ are finite.
	\end{lemma}
	\begin{proof}
	If $S^{0}$ is non-constant, take
	\[
	\Psi_{T,s,S^{0}}(S,t)=\begin{cases}
	\varphi(S,3T\times t) & t\in\left[0,\frac{1}{3}\right],\\
	\phi_{s}\left(\varphi(S,T),C\times\left(t-\frac{1}{3}\right)\right) & t\in\left[\frac{1}{3},\frac{2}{3}\right],\\
	\psi_{s}\left(\phi_{s}\left(\varphi(S,T),C/3\right),t-\frac{2}{3}\right) & t\in\left[\frac{2}{3},\frac{1}{3}\right],
	\end{cases}
	\]
	where $\psi$ is a rotation from $s$ to ${\rm e}_{1}$ with speed $3 d(s, {\rm e}_1)$.
	$C$ should be chosen such that $\tanh(C/3-2\varepsilon)$ is close
	enough to $1$, guaranteeing that we end up in the arctic.
	
	If $S^{0}$ is constant, we skip $\phi$ and take $\psi$ to be
	a rotation of $S^{0}$ to ${\rm e}_{1}$. 
	
	The speed $\partial_t \Psi$ can be calculated explicitly on each part of the path---in the time interval $[0,1/3]$ it is given by $\frac{1}{3T} \norm{DH}$ which is bounded on the entire $\Omega$. 
	During $[1/3,2/3]$ it is bounded by $\frac{\sqrt L}{C}$, since each coordinate of $\partial_t \phi(S,t)$ is bounded by $1$, hence $\norm{\partial_t \phi(S,t)}\le \sqrt L$.
	Finally, in the last interval the speed is bounded by $3\pi \sqrt L$, since during time $1/3$ each coordinate crosses distance at most $\pi$.
	
	In order to show that the entropy is bounded we use the fact that for a flow given by an equation of the type 
	\begin{align*}
	\partial_t \Phi(S,t) &= f(\Phi(S,t)),\\
	\Phi(S,0)&=S,
	\end{align*}
	the entropy production is the divergence of $f$:
	\begin{equation}
	\entropy_\Phi(S,t) = \int\limits_0^t \text{div} f(\Phi(S,u)) du,
	\end{equation}
	see, e.g., \cite[Section 8.2]{Teschl12ODE}. 
	In the first part of the path $\text{div}f$ is bounded by $3T \sup |\Delta H|$, in the second part by $NL$, and in the third part it is $0$.
	Since composing the paths results in adding the entropies, the overall entropy is bounded.
	\end{proof}

	\begin{proof}[Proof of Theorem \ref{th:periodic3}]
	Since all limit points of the deterministic dynamics are on a great
	circle, any configuration $S$ belongs to some set $U_{T,s,S^{0}}$.
	These are open sets, and $\Omega$ is compact, therefore there
	exists a finite cover $U_{1},\dots,U_{K}$. We set for $i=1,\dots,K$ the function $\Phi_i$ to be equal $\Psi_{T,s,S^0}$, for $T,s,S^0$ such that $U_i = U_{T,s,S^0}$.
	The hypotheses of Lemma \ref{lem:pathargument} are then satisfied, with finite $v$ and $\entropy$ (which do not depend on $\beta$), and $\Delta H \le 0$.
	\end{proof}

	\subsection{Explicit construction for $N=3$}
	
	When $N=3$ we are able to construct the path directly, without using a compactness argument. This enables us to obtain an explicit (though not optimal) estimate on the constant $C(L)$ in Theorem \ref{th:periodic3}.

	The path consists of three parts: first, we align the spins so that they fall close to a plane. We then contract the spins towards a well chosen pole. This second step is where the crucial difference between $O(3)$ and $O(2)$ enters. Finally, we rotate the whole configuration to get it close to ${\rm e}_1$ and into the arctic.
	
	The three steps are summarized in Lemmas~\ref{lem:plane_concentration},~\ref{lem:pole_contraction}, and~\ref{lem:final_rotation}. We will need a few preparations. The first is stating that ``if all spins are close to \emph{some} great circle, they are also close to a great circle \emph{in a fixed finite set} of great circles''.
	
	\begin{lemma}
		\label{lem:cover_sphere}
		There exists $C>0$ such that for any $\epsilon>0$, one can find $K\leq C\epsilon^{-2}$ integer, and a set $\{v_1,\dots, v_K\}\subset\bbS^{2}$ such that for any $L$, and any collection $S_1,\dots, S_L\in \bbS^2$ satisfying $|S_i\cdot s| \leq \epsilon$ for all $i\in \{1,\dots, L\}$ and some fixed $s\in \bbS^2$, there exits $k\in \{1,\dots, K\}$ such that $|S_i\cdot v_k|<2\epsilon$ for all $i\in \{1,\dots,L\}$.
	\end{lemma}
	\begin{proof}
		Let $\epsilon>0$. Let $A\subset \bbS^2$ be such that for every $s\in \bbS^2$, there is $v\in A$ such that $|s-v|<\epsilon$. One can find such a set containing at most $C\epsilon^{-2}$ points with $C$ universal (for example, by projecting on $\bbS^2$ an $\epsilon/2$ mesh-size grid on the boundary of the cube of side $2$ centred at $0$). Then for such a set $A$, for any $s\in \bbS^2$, there is $v\in A$ with $|v-s|<\epsilon$. For this $v$, one has $|S\cdot v|\leq |S\cdot (v-s)| + |S\cdot s|\leq \epsilon + |S\cdot s|$ for every $S\in \bbS^2$, so
		\begin{equation*}
			|S\cdot s |<\epsilon \implies |S\cdot v|< 2\epsilon,
		\end{equation*}which gives the claim.
	\end{proof}
	Fix $\sin(\frac{\pi}{16})\geq \epsilon>0$, and let $V_{\epsilon} = \{v_1,\dots,v_{K}\}\subset \bbS^2$ be a set whose existence is guaranteed by the previous Lemma (with $K\leq C\epsilon^{-2}$).
	
	For $s\in \bbS^2$, define then
	\begin{equation*}
		D_1(\epsilon,s) = \{S\in \Omega_L:\ |S_i\cdot s|<\epsilon,\ \forall i=1,\dots,L \}.
	\end{equation*}
	Define also
	\begin{equation*}
		\Omega_L^{\pm}(s,\epsilon) = \{S\in \Omega_L:\ S_i\cdot \pm s \geq 1-2\arcsin(\epsilon)^2\,,\; \forall i=1,\dots,L \}.
	\end{equation*}

	The three steps are then represented by the following three lemmas. For the first two of them, the precise construction and proof is given in the next subsections. 
	\begin{lemma}
		\label{lem:plane_concentration}
		The application $\varphi:\Omega_L\times [0,\tau]\to \Omega_L$ with $\tau = L -1 - \epsilon$, to be constructed in Section~\ref{sec:local_alignement}, is such that
		\begin{enumerate}
			\item $\varphi(S,0) =S$ and $\varphi(S,\tau)\in D_1(\epsilon,v)$ for some $v\in V_{\epsilon}$;
			\item $t\mapsto \varphi(S,t)$ is continuous, and piecewise differentiable;
			\item $t\mapsto H^{\periodic}_L(\varphi(S,t))$ is non-increasing;
			\item $\norm{\dot{\varphi}(S,t)}^2\leq 4\pi^2 L$;
			\item letting $f_t(S) = \varphi(S,t)$, one has $\Big|\frac{d (f_t)_{\#}\nu_L }{d\nu_L} \Big|\leq \Big(\frac{L}{\epsilon}\Big)^{L}$ for every $t\in [0,\tau]$ and $S$ in the image of $f_t$.
		\end{enumerate}
	\end{lemma}
	
	\begin{lemma}
		\label{lem:pole_contraction}
		Let $s\in \bbS^2$.
		The application $\phi_s: D_1(\epsilon,s)\times [0,\pi/2-\arcsin(\epsilon)/2] \to \Omega_L$, to be constructed in Section~\ref{sec:contraction_towards_pole}, is such that
		\begin{enumerate}
			\item $\phi_s(S,0) =S$ and $\phi_s(S,\pi/2-2\arcsin(\epsilon))\in \Omega_L^+(s,\epsilon)\cup \Omega_L^-(s,\epsilon)$;
			\item $t\mapsto \phi_s(S,t)$ is continuous, and differentiable;
			\item $t\mapsto H^{\periodic}_L(\phi_s(S,t))$ is non-increasing;
			\item $\norm{\dot{\phi}_s(S,t)}^2\leq 4\pi^2 L$;
			\item letting $f_t(S) = \phi_s(S,t)$, one    has $\Big|\frac{d (f_t)_{\#}\nu_L }{d\nu_L} \Big|\leq 2^{L}$,  for all $t\in [0,\pi/2-2\arcsin(\epsilon)]$ and $S$ in the image of $f_t$.
		\end{enumerate}
	\end{lemma}
	
	\begin{lemma}
		\label{lem:final_rotation}
		Let $s\in \bbS^2$. There exist $\tau=\tau(s)\in[0,\pi]$, and $\psi_s: \Omega_L^+(s,\epsilon)\times [0,\tau] \to \Omega_L$ such that
		\begin{enumerate}
			\item $\psi_s(S,0) =S$ and $\psi_s(S,\tau)\in \Omega_L^+({\rm e}_1,\epsilon)$;
			\item $t\mapsto \psi_s(S,t)$ is continuous, and differentiable;
			\item $t\mapsto H^{\periodic}_L(\psi_s(S,t))$ is constant;
			\item $\norm{\dot{\psi}_s(S,t)}^2\leq 4\pi^2 L$;
			\item letting $f_t(S) = \psi_s(S,t)$, one has $\Big|\frac{d (f_t)_{\#}\nu_L }{d\nu_L} \Big| = 1 $ for every $t\in [0,\tau]$ and $S$ in the image of $f_t$.
		\end{enumerate}
	\end{lemma}
	\begin{proof}[Proof of Lemma \ref{lem:final_rotation}]
		Simply set $\tau$ to be the angle between $s$ and ${\rm e}_1$, $R_t$ the rotation of angle $t$ in the plane spanned by $s,{\rm e}_1$ (with positive direction from $s$ to ${\rm e}_1$). Setting $\psi(S,t) = R_t (S)$ does the job.
	\end{proof}
	
	In the next sections, we prove Lemma~\ref{lem:plane_concentration} and~\ref{lem:pole_contraction}. We then use them together with Lemma~\ref{lem:pathargument} to bound the spectral gap in Section~\ref{sec:spectral_gap}.
	
	\subsection{Local alignment: proof of Lemma~\ref{lem:plane_concentration}}
	\label{sec:local_alignement}
	
	The path we will use will align $S_1,S_2,S_3$ by rotating $S_2$ around the axis spanned by $S_1$ to end on the geodesic between $S_1$ and $S_3$, then align $S_1,\dots,S_4$ by rotating the pair $S_2,S_3$ around the axis spanned by $S_1$ so that $S_3$ ends on the geodesic between $S_1$ and $S_4$, and so on and so forth. We will first consider a sequence of mappings which is better expressed in a suitable choice of coordinates. For $1< k\leq L+1$, consider the following coordinate system. Let $v_1 = S_1$. If $S_{k} \notin \{v_1,-v_1\} $, let
	\begin{equation*}
		v_2 = \frac{S_{k} - (S_{k} \cdot v_1) v_1}{\sqrt{1-(S_{k} \cdot v_1)^2} },
	\end{equation*}otherwise, let $v_2$ be any (fixed) norm one vector orthogonal to $v_1$. Let then $v_3$ be any vector such that $(v_1,v_2,v_3)$ is an orthonormal basis. Express then a point $s$ on the sphere as
	\begin{equation*}
		s = (s\cdot v_1) v_1 + \sqrt{1-(s\cdot v_1)^2}\cos(\theta)v_2 + \sqrt{1-(s\cdot v_1)^2}\sin(\theta)v_3
	\end{equation*}with $\theta\in [-\pi,\pi)$ the angle between $v_2$ and the projection of $s$ in the $(v_2,v_3)$-plane (so that $\theta = 0$ when $s= S_k$). Write $u_i= S_i\cdot v_1 \in[-1,1]$, and $\theta_i$ for the angle $\theta$ corresponding to $s=S_i$. With this choice of coordinates, one has
	\begin{equation*}
		\sum_{i=1}^{k-1}S_i\cdot S_{i+1} = \sum_{i=1}^{k-1} \big(u_iu_{i+1} + \sqrt{1-u_i^2}\sqrt{1-u_{i+1}^2}\cos(\theta_i-\theta_{i+1})\big),
	\end{equation*}where $\theta_k = 0$.
	
	Define then the path to be
	\begin{equation*}
		\big(\varphi_{k}(S,t)\big)_i = \begin{cases}
			S_i & \text{ if } k\leq i \leq L+1,\\
			u_i v_1 + \sqrt{1-u_i^2}\cos(\theta_i-t\theta_{k-1})v_2 + \sqrt{1-u_i^2}\sin(\theta_i-t\theta_{k-1})v_3 & \text{ if } 1< i < k.
		\end{cases}
	\end{equation*}
	
	As, in the chosen coordinates, $d\nu_L(S) = d \nu(S_1)\prod_{i=k}^L d\nu(S_i) \prod_{i=2}^{k-1}du_i d\theta_i $, one has
	\begin{equation*}
		\Big|\frac{d (f_t)_{\#}\nu_L }{d\nu_L} \Big| = \frac{1}{1-t},
	\end{equation*}where $f_t(S) = \varphi_k(S,t)$.
	
	Note then that
	\begin{align*}
		-H^{\periodic}_L(\varphi_{k}(S,t)) &= \sum_{i=1}^{k-1} \big(\varphi_{k}(S,t)\big)_i\cdot \big(\varphi_{k}(S,t)\big)_{i+1} +\sum_{i=k}^{L} S_i\cdot S_{i+1}\\
		&= \sum_{i=1}^{k-2}S_i\cdot S_{i+1} + \big(u_{k-1}u_{k} + \sqrt{1-u_{k-1}^2}\sqrt{1-u_{k}^2} \cos((1-t)\theta_{k-1})\big) +\sum_{i=k}^{L} S_i\cdot S_{i+1}.
	\end{align*}As the angles are in $[-\pi,\pi]$, $-H^{\periodic}_L(\varphi_{k}(S,t))\geq -H^{\periodic}_L(S)$.
	
	One has that $\big(\varphi_{k}(S,1)\big)_{k-1}$ is in the plane spanned by $S_1$ and $S_k$, and that the distance to that plane of $\big(\varphi_{k}(S,1-\epsilon)\big)_{k-1}$ is at most
	\begin{equation*}
		|\big(\varphi_{S}^{(k)}(1-\epsilon)\big)_{k-1} \cdot v_3|\leq \sin(\epsilon \pi).
	\end{equation*}
	
	Let $\tau_0 = 0$, $\tau_k = \tau_{k-1} + 1-\epsilon/L$. Define
	\begin{gather*}
		\varphi(S,0) = S,\\
		\varphi(S,t) = \varphi_{k}(\varphi(S,\tau_{k-1}),t-\tau_{k-1}),\ t\in [\tau_{k-1},\tau_k].
	\end{gather*}
	This defines a continuous, and piecewise differentiable function of $t\in[0,\tau_{L-1}]$, where $\tau_{L-1}=L-1-\epsilon$. By triangular inequality, the distance of $(\varphi(S,\tau_{L-1}))_i$ to the plane spanned by $S_1$ and $S_L$ is at most $\epsilon$. By the previous observations, the Hamiltonian is decreasing in $t$ and
	\begin{equation*}
		\Big|\frac{d (f_t)_{\#}\nu_L }{d\nu_L} \Big| \leq (L\epsilon^{-1})^{L},
	\end{equation*}where $f_t(S) = \varphi(S,t)$.
	
	It remains to prove the bound on $\norm{\dot{\varphi}(S,t)}^2$. On each time interval $[\tau_{k-1},\tau_k]$, the mapping is a rotation of the $k$ first spins around the axis spanned by $S_1$, so $\norm{\dot{\varphi}(S,t)}^2 \leq 4\pi^4 L$.

	\subsection{Contraction towards a pole: proof of Lemma~\ref{lem:pole_contraction}}
	\label{sec:contraction_towards_pole}
	
	For $s\in\bbS^2$, parametrize $S\in\Omega_L$ by $(\theta,\theta')\in ([-\frac{\pi}{2},\frac{\pi}{2}]\times [0,2\pi])^L$ via $S_i = \cos(\frac{\pi}{2} + \theta_i)s + \sin(\frac{\pi}{2} +\theta_i)\cos(\theta_i')v_1 + \sin(\frac{\pi}{2}+\theta_i)\sin(\theta_i')v_2$ where $(s,v_1,v_2)$ form an orthonormal basis of $\R^3$. Define the configuration $X_t=X_t(S)=X_t(\theta,\theta')$ via
	\begin{equation*}
		\big(X_t(\theta,\theta')\big)_i = \cos(\frac{\pi}{2} + \theta_i+t)s + \sin(\frac{\pi}{2} +\theta_i+t)\cos(\theta_i')v_1 + \sin(\frac{\pi}{2}+\theta_i+t)\sin(\theta_i')v_2
	\end{equation*}
	
	Using $\sin(\frac{\pi}{2} + a) = \cos(a)$ and $\cos(\frac{\pi}{2} + a) = -\sin(a)$, one obtains
	\begin{equation*}
		-H^{\periodic}_L(X_t(\theta,\theta')) = \sum_{i=1}^L \sin( \theta_i+t)\sin( \theta_{i+1}+t) +
		\cos(\theta_i+t)\cos(\theta_{i+1}+t) \cos(\theta_i'-\theta_{i+1}').
	\end{equation*}So, as $\sin a\cos b + \sin b\cos a = \sin(a+b)$,
	\begin{equation*}
		\frac{d}{dt} -H^{\periodic}_L(X_t(\theta,\theta')) = \sum_{i=1}^L \sin( \theta_{i}+\theta_{i+1}+2t) (1-\cos(\theta_i'-\theta_{i+1}')).
	\end{equation*}Moreover,
	\begin{equation*}
		\frac{d^2}{dt^2} -H^{\periodic}_L(X_t(\theta,\theta')) = 2\sum_{i=1}^L \cos( \theta_{i}+\theta_{i+1}+2t) (1-\cos(\theta_i'-\theta_{i+1}')).
	\end{equation*}
	
	First observe three things:
	\begin{enumerate}
		\item if $0\leq \theta_i +t \leq \pi/2$ for all $i$s, then $\frac{d}{dt} -H^{\periodic}_L(X_t(\theta,\theta'))\geq 0$;
		\item if $-\pi/2\leq \theta_i -t \leq 0$ for all $i$s, then $\frac{d}{dt} -H^{\periodic}_L(X_{-t}(\theta,\theta'))\geq 0$;
		\item if $|\theta_i|< \frac{\pi}{8}$ for all $i$s and $t<\frac{\pi}{8}$, $\cos( \theta_{i}+\theta_{i+1}+2t) \geq 0$ for all $i$s and so $\frac{d^2}{dt^2} -H^{\periodic}_L(X_t(\theta,\theta'))\geq 0$.
	\end{enumerate}
	
	When $S\in D_1(\epsilon,s)$, we have $|\theta_i|\leq \arcsin(\epsilon)$. Therefore, by our constraint on $\epsilon$, the sign of $\frac{d}{dt} -H^{\periodic}_L(X_t(\theta,\theta'))$ at $t=0$ determines the desired monotonicity for either $t\in [0,\pi/2 - 2\arcsin(\epsilon)]$ or $-t\in [0,\pi/2 - 2\arcsin(\epsilon)]$: indeed, the third observation above gives the monotonicity for all $|t|\leq \pi/8$, and the first two observations show the monotonicity for $|t|>\pi/8$ since $|\theta_i|\leq \arcsin(\epsilon)$.
	
	We can now construct the path $\phi_s$ as follows:
	\begin{itemize}
		\item If $S\in D_1(\epsilon,s)$ is such that $\frac{d}{dt} -H^{\periodic}_L(X_t(S))\big|_{t=0} \geq 0$, set
		\begin{equation*}
			\phi_s(S,t) = X_t(S),
		\end{equation*}for $t\in [0,\pi/2- 2\arcsin(\epsilon)]$.
		\item If $S\in D_1(\epsilon,s)$ is such that $\frac{d}{dt} -H^{\periodic}_L(X_t(S))\big|_{t=0} < 0$, set
		\begin{equation*}
			\phi_s(S,t) = X_{-t}(S),
		\end{equation*}for $t\in [0,\pi/2- 2\arcsin(\epsilon)]$.
	\end{itemize}
	As $\epsilon<\sin(\frac{\pi}{16})$, the first three points of Lemma~\ref{lem:pole_contraction} are obvious (the first one follows from $(\phi_s(S,t))_i\cdot s = \cos(\pi/2 + \theta_i \pm t)$ depending on the case). To obtain the one-before-last, observe that the transformation is a rotation of each coordinates, so $\norm{\dot{\phi}_s(S,t)}^2 \leq 4\pi^2 L$.
	
	The last point of Lemma~\ref{lem:pole_contraction}  follows from
	\begin{equation*}
		d\nu_L(S) = \prod_{i=1}^L \cos(\theta_i)d\theta_i d\theta_i',
	\end{equation*}and $\frac{\cos(\theta_i-t)}{\cos(\theta_i)}\leq \frac{1}{\cos(\pi/16)} \leq 2$ for $\theta_i,t$ in the concerned regions.
	
	\subsection{Spectral Gap}
	\label{sec:spectral_gap}
	
	Let $\epsilon>0$. Rather than parametrizing the sets $U_i$'s of Lemma~\ref{lem:pathargument} with numbers, we label them with pairs $(\sigma,v)$ with $\sigma\in \{-,+\}$ and $v\in V_{\epsilon}$, where $V_{\epsilon}$ is given as right after Lemma~\ref{lem:cover_sphere}. Explicitly, we use $U_{\sigma,v} = \{s:\ \chi_{\sigma,v}(s) =1\}$, where, for $v\in V_{\epsilon}$, $\sigma\in \{-,+\}$,
	\begin{equation*}
		\chi_{\sigma,v}(S) = \mathbf{1}_{\varphi(S,\tau_1) \in D_1(\epsilon,v)} \mathbf{1}_{\phi_v(\varphi(S,\tau_1), \tau_2)\in \Omega_L^{\sigma}(v,\epsilon)},
	\end{equation*}
	and $\tau_1=L-1-\epsilon$ and $\varphi$ are as in Lemma~\ref{lem:plane_concentration}, while $\tau_2=\frac{\pi}{2}-2\arcsin(\epsilon)$, and $\phi_s$ are as in Lemma~\ref{lem:pole_contraction}. Finally, let $\tau_3(s),\psi_s$ be as in Lemma~\ref{lem:final_rotation}, and define 
	\[
	T_1 = \tau_1, \quad T_2 = \tau_1+\tau_2, \quad T_3 = \tau_1+\tau_2+\tau_3(\sigma  v)\,.
	\] Next, introduce the path
	\begin{equation*}
		\Psi_{\sigma,v}(S,t) = \begin{cases}
			\varphi(S,t) & \text{ if } t\leq T_1,\\
			\phi_v(\varphi(S,T_1), t-T_1) & \text{ if } t\in [T_1, T_2],\\
			\psi_{v}\big(\phi_v( \varphi(S,T_1) ,T_2-T_1), t-T_2\big)& \text{ if } t\in [T_2, T_3],
		\end{cases}
	\end{equation*}
	which is defined on $U_{\sigma,v}$, is continuous, and is piecewise differentiable in $t$ with the bound
	\begin{equation}
		\label{eq:bound_path_time_derivative}
		\norm{\dot{\Psi}_{\sigma,v}(S,t)}^2 \leq 4\pi^2 L,
	\end{equation}at every point of differentiability. Moreover, for any $L$ large enough
	\begin{equation}
		\label{eq:bound_Jacobian}
		\Big| \frac{d(f_t)_{\#}\nu_L}{d\nu_L}\Big|\leq (L\epsilon^{-1})^L,
	\end{equation}where $f_t(S) = \Psi_{\sigma,v}(S,t)$. 
	Finally, from the monotonicity of the Hamiltonian along each piece of the path,
	\begin{equation}
		\label{eq:monotonicity_Ham}
		-H^{\periodic}_L(\Psi_{\sigma,v}(S,t)) \geq -H^{\periodic}_L(S).
	\end{equation}		
	Applying Lemma \ref{lem:pathargument} with these estimates on the functions $(S,t)\mapsto \Psi_{\sigma,v}(S,t/T_3)$, and choosing $\epsilon$ such that $\Omega_L^+(e_1,\epsilon)\subseteq A$, yields the bound on the relaxation time stated in Theorem \ref{th:periodic3} in the case $N=3$.
	\qed
	
	\bibliographystyle{plain}
	
	\bibliography{ON_bibliography}

\begin{thebibliography}{10}

\bibitem{bakrygentilledoux2014}
Dominique Bakry, Ivan Gentil, and Michel Ledoux.
\newblock {\em Analysis and geometry of {M}arkov diffusion operators}, volume
  348 of {\em Grundlehren der mathematischen Wissenschaften [Fundamental
  Principles of Mathematical Sciences]}.
\newblock Springer, Cham, 2014.

\bibitem{becker2020spectral}
Simon Becker and Angeliki Menegaki.
\newblock Spectral gap in mean-field ${O}({N})$-model.
\newblock {\em Communications in Mathematical Physics}, 380(3):1361--1400,
  2020.

\bibitem{bray1988dynamics}
AJ~Bray.
\newblock Dynamics of dilute magnets above ${T}_c$.
\newblock {\em Physical review letters}, 60(8):720, 1988.

\bibitem{cosco2021topologically}
Cl{\'e}ment Cosco and Assaf Shapira.
\newblock Topologically induced metastability in a periodic {XY} chain.
\newblock {\em Journal of mathematical physics}, 62(4), 2021.

\bibitem{dyson1956general}
Freeman~J Dyson.
\newblock General theory of spin-wave interactions.
\newblock {\em Physical review}, 102(5):1217, 1956.

\bibitem{EellsLemaire83harmonicmaps}
James Eells and Luc Lemaire.
\newblock {\em Selected topics in harmonic maps}, volume~50 of {\em CBMS
  Regional Conference Series in Mathematics}.
\newblock Conference Board of the Mathematical Sciences, Washington, DC; by the
  American Mathematical Society, Providence, RI, 1983.

\bibitem{friedli_velenik_2017}
Sacha Friedli and Yvan Velenik.
\newblock {\em Statistical Mechanics of Lattice Systems: A Concrete
  Mathematical Introduction}.
\newblock Cambridge University Press, 2017.

\bibitem{frohlich1983spin}
J{\"u}rg Fr{\"o}hlich and Charles-Edouard Pfister.
\newblock Spin waves, vortices, and the structure of equilibrium states in the
  classical {XY} model.
\newblock {\em Communications in Mathematical Physics}, 89:303--327, 1983.

\bibitem{frohlich1981kosterlitz}
J{\"u}rg Fr{\"o}hlich and Thomas Spencer.
\newblock The {K}osterlitz-{T}houless transition in two-dimensional abelian
  spin systems and the coulomb gas.
\newblock {\em Communications in Mathematical Physics}, 81(4):527--602, 1981.

\bibitem{garban2023statistical}
Christophe Garban and Avelio Sep{\'u}lveda.
\newblock Statistical reconstruction of the {GFF} and {KT} transition.
\newblock {\em Journal of the European Mathematical Society}, 2023.

\bibitem{hairer2009introduction}
Martin Hairer.
\newblock An introduction to stochastic {PDE}s.
\newblock {\em arXiv preprint arXiv:0907.4178}, 2009.

\bibitem{kolesnikov2017brascamp}
Alexander~V Kolesnikov and Emanuel Milman.
\newblock Brascamp--lieb-type inequalities on weighted riemannian manifolds
  with boundary.
\newblock {\em The Journal of Geometric Analysis}, 27:1680--1702, 2017.

\bibitem{kosterlitz1973ordering}
JM~Kosterlitz and DJ~Thouless.
\newblock Ordering, metastability and phase transitions in two-dimensional
  systems.
\newblock {\em Journal of Physics C Solid State Physics}, 6(7):1181--1203,
  1973.

\bibitem{martinelli1999lectures}
Fabio Martinelli.
\newblock Lectures on {G}lauber dynamics for discrete spin models.
\newblock {\em Lectures on probability theory and statistics (Saint-Flour
  1997), Springer LNM}, 1717:93--191, 1999.

\bibitem{mermin1979topological}
N~David Mermin.
\newblock The topological theory of defects in ordered media.
\newblock {\em Reviews of Modern Physics}, 51(3):591, 1979.

\bibitem{newman2018gaussian}
Charles~M Newman and Wei Wu.
\newblock Gaussian fluctuations for the classical {XY} model.
\newblock In {\em Annales de l'institut Henri Poincare (B) Probability and
  Statistics}, volume~54, pages 1759--1777. Institute of Mathematical
  Statistics, 2018.

\bibitem{peled2019lectures}
Ron Peled and Yinon Spinka.
\newblock Lectures on the spin and loop ${O}({N})$ models.
\newblock In {\em Sojourns in Probability Theory and Statistical Physics-I:
  Spin Glasses and Statistical Mechanics, A Festschrift for Charles M. Newman},
  pages 246--320. Springer, 2019.

\bibitem{sinclair1992improved}
Alistair Sinclair.
\newblock Improved bounds for mixing rates of markov chains and multicommodity
  flow.
\newblock {\em Combinatorics, probability and Computing}, 1(4):351--370, 1992.

\bibitem{Teschl12ODE}
Gerald Teschl.
\newblock {\em Ordinary differential equations and dynamical systems}, volume
  140 of {\em Graduate Studies in Mathematics}.
\newblock American Mathematical Society, Providence, RI, 2012.

\end{thebibliography}
	
\end{document}